\documentclass[pdflatex,sn-mathphys-num]{sn-jnl}

\usepackage{graphicx}%
\usepackage{multirow}%
\usepackage{amsmath,amssymb,amsfonts}%
\usepackage{amsthm}%
\usepackage{mathrsfs}%
\usepackage[title]{appendix}%
\usepackage{xcolor}%
\usepackage{textcomp}%
\usepackage{manyfoot}%
\usepackage{booktabs}%
\usepackage{algorithmicx}%
\usepackage{algpseudocode}%
\usepackage{listings}%
\usepackage{lineno}
\modulolinenumbers[1]

\usepackage{graphicx}
\usepackage[T1]{fontenc}
\graphicspath{{figs/}}
\usepackage{amsmath}
\usepackage{amsfonts}
\usepackage{xcolor}
\usepackage{contour}

\usepackage{amssymb}
\usepackage{amsthm}
\usepackage{overpic}
\usepackage[ruled,vlined,linesnumbered]{algorithm2e}
\usepackage{appendix}
\usepackage{comment}
\newtheorem{theorem}{Theorem}[section]
\newtheorem*{theorem*}{Theorem}

\newtheorem{proposition}[theorem]{Proposition}
\newtheorem{remark}[theorem]{Remark}
\newtheorem{problem}[theorem]{Problem}
\newtheorem{lemma}[theorem]{Lemma}
\newtheorem{corollary}[theorem]{Corollary}

\theoremstyle{definition}

\newtheorem{definition}[theorem]{Definition}
\newtheorem{example}[theorem]{Example}
\newtheorem*{definition*}{Definition}
\newtheorem*{example*}{Example}

\usepackage{anysize}
\usepackage{url}

\usepackage{setspace}
\usepackage{longtable}

\usepackage{pgf,tikz}
\usetikzlibrary{arrows}

\long\def\comment#1\endcomment{}

\begin{document}

\title[Families of CPRC surfaces and Plateau's problem]{Families of surfaces with constant ratio of principal curvatures and Plateau's problem}


\author*[1]{\fnm{Mikhail} \sur{Skopenkov}}\email{mikhail.skopenkov@gmail.com}

\author[1]{\fnm{Khusrav} \sur{Yorov}}


\affil[1]{\orgdiv{CEMSE}, \orgname{King Abdullah University of Science and Technology}, \orgaddress{ \city{Thuwal}, 
\country{Saudi Arabia}}}



\abstract{
This work is on surfaces with a constant ratio of principal curvatures. These CRPC surfaces generalize minimal surfaces but are much more challenging to construct. We propose a construction of a family of such surfaces containing a given minimal surface without flat points. This leads to a partial solution of Plateau's problem for CRPC surfaces. We obtain analogous results in isotropic geometry. This work illustrates a general approach to solving Euclidean problems by starting with their isotropic analogs. Besides, we apply the method of successive approximations and analytic majorization.  
}


\maketitle

\section{Introduction}\label{sec-intro}

This work is on 
surfaces with a constant ratio of principal curvatures. These \emph{CRPC surfaces} generalize minimal surfaces but are much more challenging to construct. We therefore propose a construction of a family of such surfaces containing a given minimal surface (Theorem~\ref{thm-euclidean}). This leads to a partial solution of Plateau's problem, 
which is to span a given curve by a CRPC surface (Corollary~\ref{thm-plateau-euc} and Fig.~\ref{fig:5}, left).


These surfaces are a particular case of Weingarten surfaces, a classical subject of differential geometry. \emph{Weingarten surfaces} are those whose principal curvatures satisfy a fixed functional relation 
(essentially equivalent to a fixed relation 
between the mean curvature $H$ and the Gaussian curvature $K$). 
Within this class, CRPC surfaces can serve as an archetype for instances 
of Plateau's problem inaccessible to known methods. They have a 
geometric characterization as surfaces with a constant angle between asymptotic curves (in the case when $K<0$).

\begin{figure}[bh]
    \centering
\begin{overpic}
[width=0.37\linewidth]{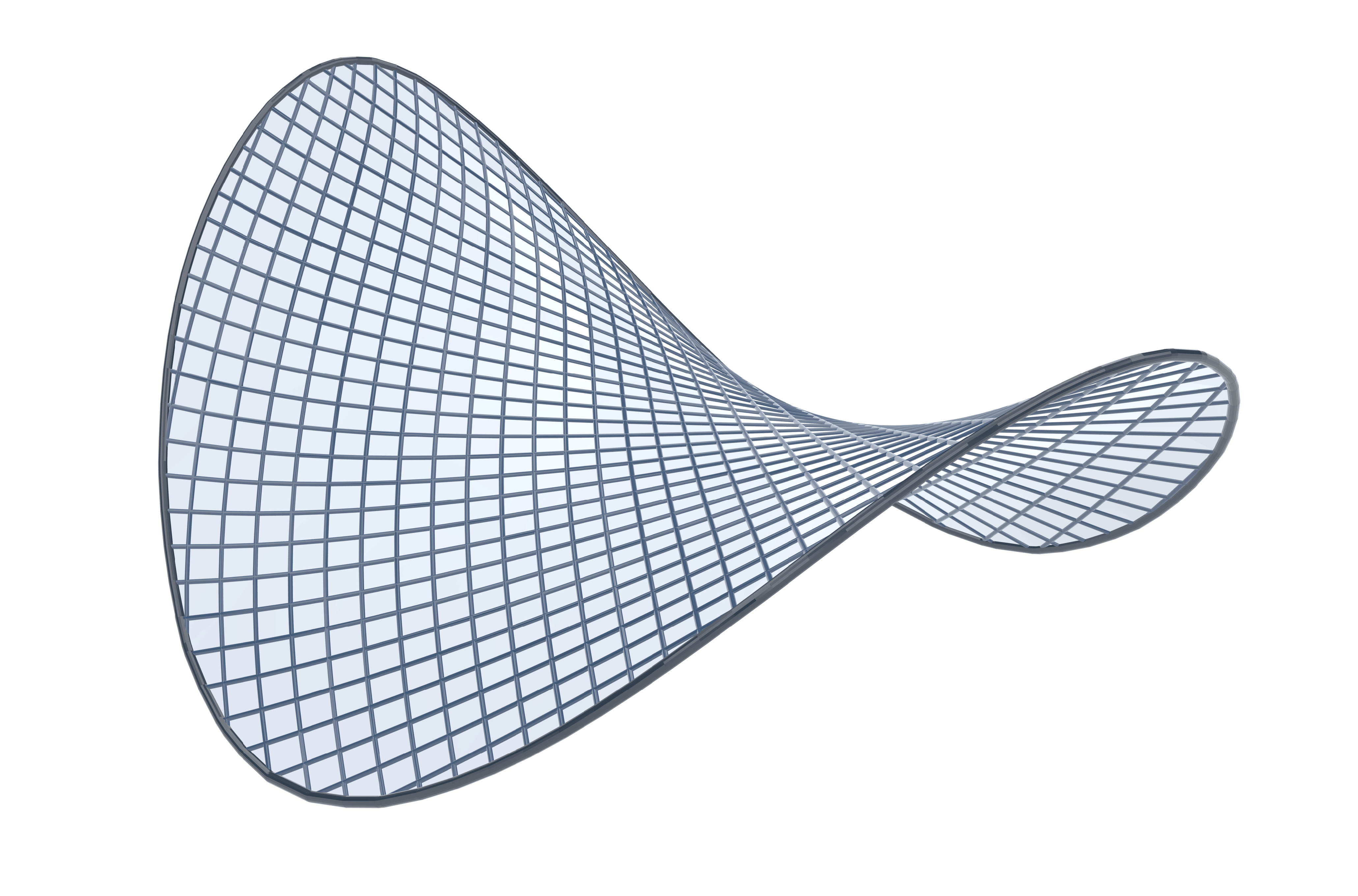}
\small
 \end{overpic}
\qquad
\begin{overpic}
[width=0.43\linewidth]{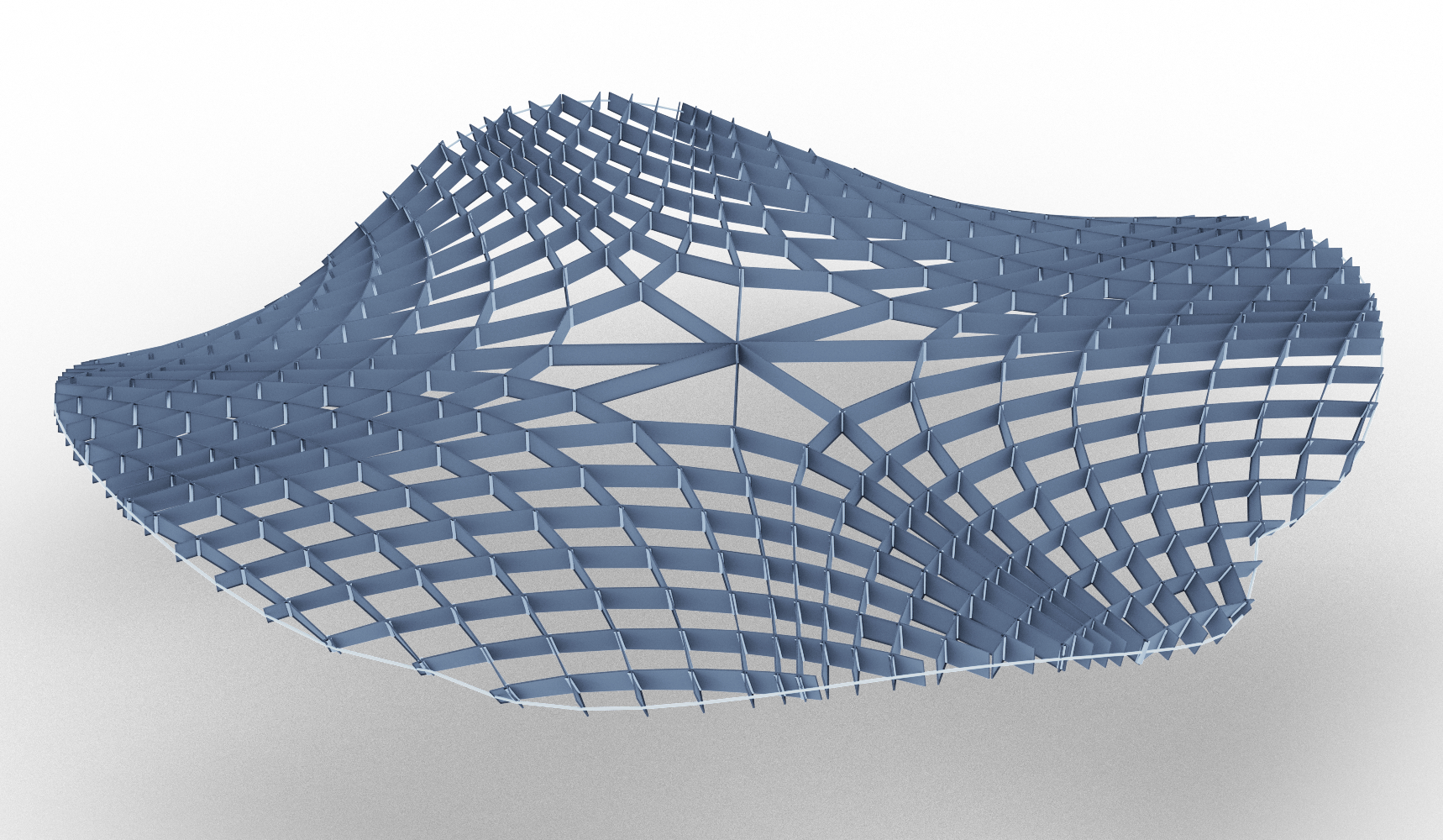}
 \end{overpic}
\caption{Left: A surface with a constant ratio of principal curvatures spanning a smooth curve. The asymptotic curves (dark) intersect at a constant angle ($80^\circ$). Right: An asymptotic gridshell obtained by bending originally rectangular strips and placing them orthogonal to a CRPC surface. The strips follow the asymptotic curves and intersect each other at a constant angle ($60^\circ$) except for a flat point 
in the middle. The gridshell was computed by optimizing an approximate CPRC surface.}
\label{fig:5}
\end{figure}

\begin{figure}[htb]
    \centering
\begin{overpic}
[width=0.4\linewidth]{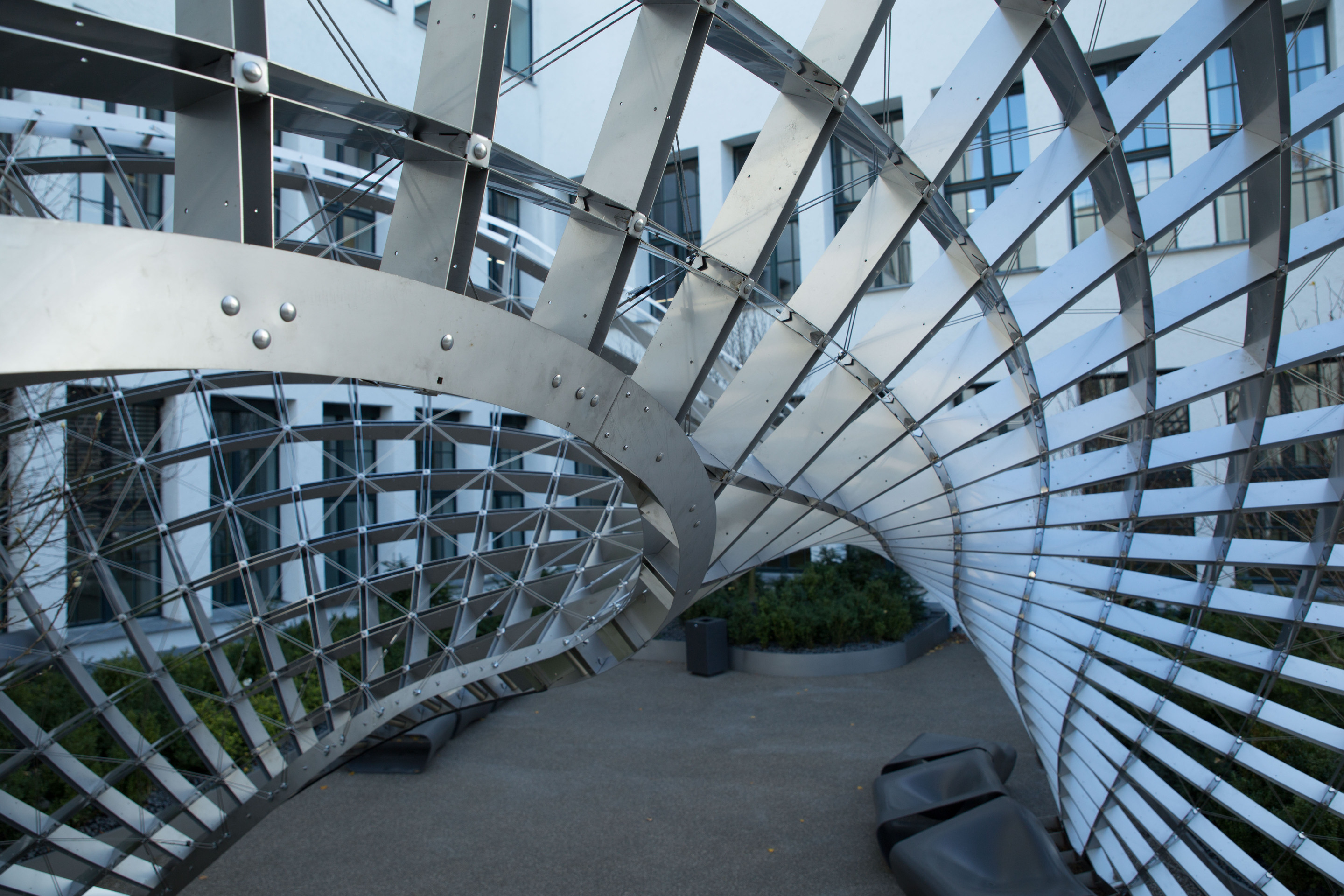}
\small
 \end{overpic}
\caption{An asymptotic gridshell by E.~Schling \cite{schling:2018}. The strips follow the asymptotic lines of a minimal surface and therefore intersect each other at a right angle.
}
\label{fig:4}
\end{figure}

Recent interest in CRPC surfaces extends beyond classical differential geometry to architectural freeform design.  One of the inspiring designs is asymptotic gridshells by E.~Schling \cite{schling:2018,schling-aag-2018}; see Fig.~\ref{fig:4}. They are built from rectangular strips. The strips are bent and placed orthogonal to the reference surface, making them follow the asymptotic curves. We therefore get two intersecting families of strips. In \cite{schling:2018,schling-aag-2018} the intersection angle is fixed at $90^\circ$, which is only possible for minimal surfaces and limits the range of shapes. On CRPC surfaces, the strips intersect at an arbitrary constant angle; therefore, joints remain identical, keeping fabrication simple, while the design space becomes significantly larger. See Fig.~\ref{fig:5}, right. 

This led to the question of whether \emph{the right angle between the asymptotic curves on a minimal surface can always be continuously changed to another constant angle while keeping the topology of the asymptotic net}. This was confirmed by numerical experiments by H.~Wang and H.~Pottmann 
\cite{HelmutHui2022}, and some of the resulting beautiful figures appeared in the Berlin Mathematical Calendar 2023.

\subsection{Contributions} \label{sec-contributions}

In this work, we contribute to this problem. Let us introduce the 
setup and state the main result. 

To minimize technicalities, we restrict ourselves to curves, surfaces, and families that are analytic up to the boundary. We define
an \emph{analytic Jordan curve} in $\mathbb{R}^n$ as the image of a periodic injective real analytic map $\mathbb{R}\to\mathbb{R}^n$ with nowhere vanishing derivative. Let $\Omega\subset\mathbb{R}^2$ be a domain bounded by an analytic Jordan curve in $\mathbb{R}^2$, and let $\overline\Omega$ be its closure.  A map $f\colon\overline{\Omega}\to\mathbb{R}^n$ that extends to a real analytic function in a neighborhood of $\overline{\Omega}$ is called \emph{real analytic function in the closed Jordan domain}.
If the map $f$ is injective and has a nondegenerate differential in $\overline{\Omega}$, then the image $f(\overline{\Omega})$ is called \emph{analytic surface with analytic boundary}. We mainly focus on surfaces that are graphs of real analytic functions $f\colon\overline{\Omega}\to\mathbb{R}$.
A map $f\colon\overline{\Omega}\times [a;b]\to\mathbb{R}^n$ that extends to a real analytic function in a neighborhood of $\overline{\Omega}\times [a;b]$ in $\mathbb{R}^3$ is 
a \emph{real analytic family of functions}. If for each $t\in [a;b]$, the restriction of $f$ to $\overline{\Omega}\times \{t\}\to\mathbb{R}^n$ is injective and has a nondegenerate differential, then the collection of images $f(\overline{\Omega}\times \{t\})$ is 
an \emph{analytic family of surfaces}. If $[c;d]\subset [a;b]$, then the restriction of the map $f$ to $\overline{\Omega}\times [c;d]$ is called \emph{restriction of the family}. 

Intersection points of analytic curves and surfaces come with multiplicities. If $(x(t),y(t),z(t))$ is a regular parametrization of an analytic curve intersecting the plane $ax+by+cz+d=0$ for $t=0$ (but not contained in the plane), then the maximal $m$ such that $i$-th derivative $ax^{(i)}(0)+by^{(i)}(0)+cz^{(i)}(0)=0$ for each $i=1,\dots,m-1$ is called the \emph{multiplicity} of the intersection point. 

A surface in $\mathbb{R}^3$ 
has a \emph{constant ratio $a$ of principal curvatures}, or is a \emph{CRPC surface}, if the principal curvatures $\kappa_1,\kappa_2\neq 0$ and $\kappa_1/\kappa_2=a$ or $\kappa_2/\kappa_1=a$ at each point. 
This is equivalent to $H^2/K=(a+1)^2/(4a)=\mathrm{const}$. A \emph{flat point} is one in which $\kappa_1=\kappa_2=0$. 

\begin{theorem}\label{thm-euclidean}
Each minimal surface $\Phi^0\subset \mathbb{R}^3$ with analytic boundary and no flat points (even on~$\partial\Phi^0$), which is the graph of a real analytic function in a closed Jordan domain, is contained in a unique (up to restriction) analytic family of surfaces $\Phi^s$ 
with the ratio~$s-1$ of principal curvatures and 
$\partial\Phi^s=\partial\Phi^0$.
\end{theorem}


\begin{corollary}[Partial solution to Plateau's problem for CRPC surfaces] \label{thm-plateau-euc} 
     Assume that an analytic Jordan curve in~$\mathbb{R}^3$ has 
     no more than four 
     common points (counting multiplicities) with any plane and no more than two common points (counting multiplicities) with any $z$-parallel plane. Then, for all sufficiently small~$s$, the curve can be spanned by an analytic surface with the ratio~$s-1$  of principal curvatures. 
\end{corollary}

The solution to Plateau's problem given by Corollary~\ref{thm-plateau-euc} is \emph{never} unique because both $\Phi^s$ and $\Phi^{s/(s-1)}$ have the same ratio~$s-1$ of principal curvatures. In general, Plateau's problem for CRPC surfaces may have 
uncountably many solutions (see Example~\ref{ex-catenoid} below).




The absence of flat points is essential in Theorem~\ref{thm-euclidean}. 
A minimal surface with a ``nonmultiple'' flat point is \emph{never} contained in a real analytic family of CRPC surfaces; the family necessarily acquires singularities (see Example~\ref{ex-no-xxxt} below).  
However, those singularities are mild, and it would be interesting to establish the existence of a family with reasonable regularity, especially when the minimal surface has a nontrivial topology. We conjecture that a minimal surface with flat points is still contained in a continuous family of twice differentiable CPRC surfaces if we do not fix the boundary. 

In this work, we address the case of surfaces with flat points just briefly, by embedding them into a 
family of \emph{approximate} CRPC surfaces. Since the approximation quality can be further improved by 
optimization, this is sufficient for applications such as the design of asymptotic gridshells in Fig.~\ref{fig:5}, right. 

\subsection{Main idea}

We advocate a general approach to difficult problems in Euclidean geometry, which is \emph{to start with the same problem in isotropic geometry}.
Isotropic geometry (or, more precisely, simply isotropic geometry)
is one of the classical non-Euclidean geometries, based on the group of spatial affine transformations preserving the \emph{isotropic semi-norm} $\|(x,y,z)\|_i:=\sqrt{x^2+y^2}$ \cite{sachs}.
It can be viewed as a structure-preserving simplification of Euclidean geometry. The solution in isotropic geometry is usually simpler, has fewer technical details, and can often be adapted to the solution of the original problem. 

The 
results stated above were derived exactly in this way. First, we proved their analogs in isotropic geometry (see Theorem~\ref{thm-isotropic-equivalent} and Corollary~\ref{thm-plateau-iso} below) and then modified the arguments 
for the Euclidean case. \emph{Isotropic minimal surfaces} are the graphs of harmonic functions, and \emph{isotropic CPRC surfaces} are the graphs of functions satisfying the nonlinear PDE (depending on a parameter $t$; see Section~\ref{sec-isotropic}) 
\begin{equation}\label{eq-isotropic-crpc}
    f_{xx} + f_{yy} = t \sqrt{(f_{xy})^2 - f_{xx}f_{yy}}.
\end{equation}

To construct isotropic CPRC surfaces, we reincarnate the classical methods of \emph{successive approximations} and \emph{analytic majorization}. 
We search for a solution of~\eqref{eq-isotropic-crpc} in the form of series
\begin{equation}\label{eq-expansion}
  f(x,y,t)=f^{(0)}(x,y)+tf^{(1)}(x,y)+\frac{t^2}{2}f^{(2)}(x,y)+\dots+
  \frac{t^m}{m!}f^{(m)}(x,y)+\dots.
\end{equation}
Here $f^{(0)}(x,y)$ is a harmonic function whose graph represents a given isotropic minimal surface. The functions $f^{(m)}(x,y)$ are found recursively, 
solving an infinite system of Poisson equations~\eqref{eq-poisson} obtained by differentiating~\eqref{eq-isotropic-crpc} $m$ times with respect to $t$ at $t=0$. Bounding the H\"older 
norms of solutions by the Taylor coefficients of specially constructed analytic function~\eqref{eq-l-reccurence}, we establish the convergence of series~\eqref{eq-expansion} for small $t$. Thus, we get a solution to~\eqref{eq-isotropic-crpc}. 
The approximate CRPC surfaces mentioned above are obtained by the truncation of series~\eqref{eq-expansion} at $m=2$.
The Euclidean case is 
parallel, just 
more technical.

\subsection{Previous work}


\subsubsection*{Minimal surfaces}
Classical Plateau's problem asks for a surface of minimal area spanning a given simple curve \(\Gamma\) in $\mathbb{R}^3$.
There are remarkable books~\cite{Gilbarg-Trudinger-83,nitsche-1975,dierkes-1992-II} on this subject, and we focus on a few of the most relevant results here.

The modern theory of minimal surfaces begins with the celebrated works 
by Bernstein, Korn, and M\"untz \cite{Bernstein-27, Korn-09, Muentz-11}.  
M\"untz constructed minimal surfaces by the method of successive approximations, starting from the graph of a harmonic function. The latter is a minimal surface in isotropic geometry. So, this was an early example of the ``start with the same problem in isotropic geometry'' approach, despite M\"untz did not mention isotropic geometry.
Bernstein introduced a general result, sometimes called \emph{Bernstein's Fundamental Lemma} \cite[Theorem~B]{Bernstein-27}, on the existence of an analytic family of solutions to a nonlinear elliptic PDE depending on a parameter. Theorems~\ref{thm-euclidean} and~\ref{thm-isotropic-equivalent} can be viewed as variations of this result.
Note that these theorems assert analyticity up to the boundary, not guaranteed by Bernstein's lemma. 

Plateau's problem was essentially solved by Douglas \cite{Douglas31} and Rad\'o \cite{Rado30}, who proved that any rectifiable Jordan curve \(\Gamma\) in \(\mathbb{R}^3\) can be spanned by a minimal surface. Hardt and Simon \cite{HardtSimon79} strengthened this result by showing that any finite collection of disjoint \(C^{1+\alpha}\) Jordan curves with \(0<\alpha<1\) can be spanned by a compact embedded minimal surface of class \(C^{1+\alpha}\) up to the boundary, where \(C^{k+\alpha}\) denotes the H\"older space (cf.~Definition~\ref{def-holder} below).  
If $\Gamma$ is analytic, then the surface is analytic up to the boundary~\cite[\S334]{nitsche-1975}. The analyticity of smooth solutions of general non-linear elliptic systems of PDEs was proved by Friedman \cite{Friedman-58} and Morrey--Nirenberg \cite{Morrey-Nirenberg-1957}, generalizing earlier results by Bernstein and Petrowski \cite{petrowsky-1996-i}. 
Finn's criterion \cite[Theorem 13.14]{Gilbarg-Trudinger-83} asserts that we can guarantee that the minimal surface is a graph of a function if the projection of \(\Gamma\) to the $xy$-plane is the boundary of a convex domain. 

\subsubsection*{CMC surfaces}

A natural generalization is surfaces with constant mean curvature $H$ (CMC surfaces).

Serrin~\cite[Theorem~1]{Serrin70} proved an analog of Finn's criterion for CMC surfaces. Namely, if \(\Gamma\) is a \(C^{2+\alpha}\) curve and its projection to the $xy$-plane is a Jordan curve with inward curvature at least $2H$, then there is a function whose graph is a surface with constant mean curvature $H$ spanning \(\Gamma\). The solution is unique if it exists. For boundary refinements and further developments, see the survey by L{\'o}pez~\cite{LopezBook13}.

Hildebrandt~\cite{Hildebrandt70} solved Plateau's problem for CMC surfaces with sufficiently small $H$. More precisely, if \(\Gamma\) lies in a ball of radius \(r\), then for any \(|H|\le 1/r\) there exists a surface with constant mean curvature $H$ spanning \(\Gamma\), and this bound is sharp. An analogous result holds when \(\Gamma\) lies in a cylinder; see~\cite{GulliverSpruck71}.

Koiso~\cite{Koiso2002} showed that, under suitable technical restrictions, any CMC surface is contained in a non-trivial family of CMC surfaces with the same boundary. 
This result is similar to Theorem~\ref{thm-euclidean}. However, the proof in the CMC case is much simpler than in the CRPC case: just one application of Schauder's estimates is sufficient, without the need for successive approximations and analytic majorization.

\subsubsection*{CGC surfaces}

Surfaces with constant Gaussian curvature $K$ 
are described by the nonlinear Monge--Amp\`ere equation. 
Such equations are an area of active research \cite{Figalli-2017, caffarelli-cabre-1995}.
We mention just one result by Guan and Spruck~\cite{GuanSpruck2002}:
for any positive number $K$, if  \(\Gamma\) is spanned by a \(C^{2}\) locally convex immersed surface (and locally strictly convex along its boundary) with the Gaussian curvature \emph{at least} \(K\), then it 
can be spanned by a locally strictly convex immersed surface with constant Gaussian curvature \emph{exactly} \(K\). 

\subsubsection*{Weingarten surfaces}

Caffarelli, Nirenberg, and Spruck \cite{caffarelli-1988} studied Plateau's problem for Weingarten surfaces satisfying the general equation 
$\phi(\kappa_1,\kappa_2)=\psi(x,y)$, where $\kappa_1$ and $\kappa_2$ are the principal curvatures, 
$\phi$ and $\psi$ are given functions. They proved the existence and uniqueness of a solution (among the graphs of functions) under certain convexity assumptions, including $\kappa_1\phi_{\kappa_1}+\kappa_2\phi_{\kappa_2}>0$. This does not apply to CRPC surfaces because any function of the form $\phi(\kappa_1,\kappa_2)=f(\kappa_1/\kappa_2)$ satisfies $\kappa_1\phi_{\kappa_1}+\kappa_2\phi_{\kappa_2}=0$. Ivochkina, Lin, and Trudinger \cite{lin-1994} studied the equation $H/K=\psi(x,y)$ and generalizations. See a survey by Jiao and Sun~\cite{jiao-sun-2022}.

\subsubsection*{CRPC and related surfaces}

Surfaces with a constant ratio of principal curvatures are much less studied, and Plateau's problem for them is widely open. They are archetypal for exploring the areas where previously known methods break down. CRPC surfaces are described by the slightly more complicated non-linear equation $H^2/K=\mathrm{const}$ than the Monge--Amp\`ere equation.
The only known closed-form examples distinct from minimal surfaces are 
the rotational and helical CRPC surfaces \cite{hopf-1951,kuehnel-2013,mladenov+2003,mladenov+2007,lopez-pampano-2020,HelmutHui2022, Liu-23}. In~\cite{riveros+2012,riveros+2013,staeckel-1896}, rotational CRPC surfaces with $K<0$ have been characterized via isogonal asymptotic parameterizations. 

A surface is called \emph{linear Weingarten} if there is a fixed linear relation between the two principal curvatures 
at each point. Recently, L{\'o}pez and P{\'a}mpano~\cite{lopez-pampano-2020} have classified all rotational linear Weingarten surfaces, which are CRPC surfaces when the intercept of the relation is zero. Moreover, it has been shown that linear Weingarten surfaces are rotational if foliated by a family of circles \cite{lopez-2008}. 


In isotropic geometry, the simplest CRPC surfaces 
are paraboloids with $z$-parallel axis; for them, both $H$ and $K$ are constant. In our previous work~\cite{Yorov-Pottmann-Skopenkov-23}, we classified helical, ruled, channel, and translational isotropic CRPC surfaces (the latter --- in the case when one of the parameter lines is planar). Those surfaces come in families, discussed in detail in the next section.  
Surfaces with a constant ratio of the isotropic mean and Gaussian curvatures are studied in \cite{isometric-isotropic} in terms of infinitesimal isometries.

Surfaces with a linear relation between 
$H$ and $K$ in hyperbolic geometry are surveyed in~\cite{lopez-2008}. Gálvez, Martínez, and Milán~\cite{galvez-2004} proved the existence of such surfaces spanning a planar Jordan curve.



\subsection{Organization of the paper}

The paper is organized as follows. Section~\ref{sec-proofs-example} collects closed-form examples
of families of CRPC surfaces containing known minimal surfaces, both in Euclidean and isotropic geometry. This section also contains a brief introduction to the latter. 
Sections~\ref{sec-proofs-smooth}–\ref{sec-euc} present the proofs of main results in isotropic and Euclidean geometry 
respectively. The proofs are essentially the same, but the former contains fewer technicalities and thus might give better insight. Also, Section~\ref{sec-proofs-smooth} introduces necessary tools from PDEs. Formally, Sections~\ref{sec-proofs-example}--\ref{sec-euc} are independent (except for using a few common simple lemmas) so that a reader experienced in PDEs can directly proceed to the proof of the main results in Section~\ref{sec-euc}.
Section~\ref{sec-proofs-variations} computes the initial terms of \eqref{eq-expansion} to obtain a closed-form second-order approximation to an isotropic CPRC surface. 
In a separate publication~\cite{yorov2025}, we use this approximation to initialize numerical optimization toward a Euclidean CRPC surface with a prescribed ratio of principal curvatures, boundary, and flat points, enabling asymptotic gridshells with a constant node angle (Fig.~\ref{fig:5}, right).

\section{Closed-form examples}\label{sec-proofs-example}

We begin with a few known closed-form examples of families of CRPC surfaces containing minimal surfaces. In the process, we give an exposition of the most relevant known results on CRPC surfaces. 
In this section, we understand 
families of surfaces informally, in a broader sense than in Section~\ref{sec-intro}.

\subsection{Euclidean CRPC surfaces}

The only known non-minimal CRPC surfaces written in a closed form are rotational and helical ones. They form two families that contain the simplest minimal surfaces, the catenoid and the helicoid.

\begin{example} \label{ex-catenoid}
In the cylindrical coordinate system $(r,\phi,z)$, consider 
the (half of the infinite) catenoid
$
z=\mathrm{arcosh}\,r,
$
where 
$
r\ge 1.
$
Then the catenoid is contained (as the limiting case
$a\to -1$) in the analytic family of unbounded rotational surfaces with the same boundary and the ratio $a$ of principal curvatures:
$$
z=\int_1^{r}\frac{d\rho}{\sqrt{\rho^{-2a}-1}}=\frac{r^{1+a}}{1+a}
    \, {}_2F_1\left(\frac{1}{2},\frac{1}{2}+\frac{1}{2 a}; \frac{3}{2}+\frac{1}{2 a}; r^{2a}\right)
    - \frac{1}{2a}B\left(\frac{1}{2},\frac{1}{2}+\frac{1}{2 a}\right),
    \quad\, r\ge1, \,
    a\in\left(-1;-\frac{1}{3}\right).
$$
Here ${}_2F_1(\alpha,\beta;\gamma;z)$ is 
the Gauss hypergeometric function, and $B(\alpha,\beta)$ is the beta function. The integral formula is valid for any $a<0$, but the right side 
requires $a\ne -1,-1/3,-1/5,\dots$. See~\cite[Eq.~(4)]{Yorov-Pottmann-Skopenkov-23}.

For $0<a\le 1$, there is a rotational $C^2$ surface with the same boundary and the ratio $a$ of principal curvatures, depending on an additional parameter $c\ge 1$:
$$
z=\int_r^{1}\frac{d\rho}{\sqrt{c\rho^{-2/a}-1}}, \qquad 0\le r\le1, \quad 0<a\le 1,\quad c\ge 1.
$$
Thus, for each $a\in (0;1]$, Plateau's problem for the unit circle 
has uncountably many solutions. For instance, any spherical cap spanning the 
circle has the ratio $a=1$ of principal curvatures. 
For $a\ne 1$, the surface has a flat point on the $z$-axis and is not analytic at the point unless $a=1,1/3,1/5,\dots$.
\end{example}

The closed-form formula for the family is quite involved and contains hypergeometric functions. In the case of a helicoid, the expression is even more involved (see \cite{Liu-23}).
Notice that we cannot embed a helicoid in a family of ruled CRPC surfaces rather than helical ones because there are no such surfaces besides the helicoid itself \cite[Proposition~15]{Yorov-Pottmann-Skopenkov-23}. 

Thus, we turn to isotropic geometry, where many explicit examples are available, and we get insight.


\subsection{Isotropic Geometry}
\label{sec-isotropic}

Recall that the \emph{isotropic semi-norm} in space with coordinates $x,y,z$ is defined as $\|(x,y,z)\|_i := \sqrt{x^2+y^2}$. An affine transformation of $\mathbb{R}^3$ that scales the isotropic semi-norm by a constant factor has the form
\[
\mathbf{x}' = A\cdot \mathbf{x} + \mathbf{b}, \quad
A = \begin{pmatrix}
\pm h_1 & \mp h_2 & 0\\[2pt]
h_2 & h_1 & 0\\[2pt]
c_1 & c_2 & c_3
\end{pmatrix},
\]
for some parameters $\mathbf{b}\in\mathbb{R}^3$ and $h_1,h_2,c_1,c_2,c_3\in\mathbb{R}$. These transformations form the 8-parameter group $G^8$ of \emph{general isotropic similarities}. The \emph{isotropic congruences} form the 6-parameter subgroup defined by
\[
h_1=\cos\phi, \quad h_2=\sin\phi, \quad c_3=1.
\]
These transformations appear as Euclidean congruences in the projection onto the plane $z=0$, known as the \emph{top view}. \emph{Isotropic distances} between points and \emph{isotropic angles} between lines are defined as, respectively, Euclidean distances and angles in the top view. See \cite{strubecker:1941,strubecker:1942,strubecker:1942a,sachs} for a comprehensive theory.

In isotropic differential geometry, we restrict ourselves to \emph{admissible surfaces}, that is, having no $z$-parallel tangent planes. Such surfaces can be locally represented as graphs of functions $z=f(x,y)$.
It is natural to quantify curvature in a given direction using a second-order invariant (under isotropic congruences) vanishing for planes. Thus, \emph{isotropic normal curvature} in a tangent direction $\mathbf{t}=(t_1,t_2,t_3)$ with $\|\mathbf{t}\|_i= t_1^2+t_2^2=1$ is defined as the second directional derivative of $f$:
\[
\kappa_n(\mathbf{t})=(t_1,t_2)\cdot
\begin{pmatrix}
f_{xx} & f_{xy}\\[2pt]
f_{yx} & f_{yy}
\end{pmatrix}
\cdot
\begin{pmatrix}
t_1\\[2pt]
t_2
\end{pmatrix},
\]
and \emph{isotropic shape operator} is the Hessian $\nabla^2(f)$ of $f$. Its eigenvalues $\kappa_1, \kappa_2$ are \emph{isotropic principal curvatures}. 
The \emph{isotropic mean} and \emph{Gaussian curvatures} are defined by
\begin{equation}\label{eq-curvatures}
 H:=\frac{\kappa_1+\kappa_2}{2}=\frac{f_{xx}+f_{yy}}{2}, \quad K:=\kappa_1\kappa_2=f_{xx}f_{yy}-f_{xy}^2.   
\end{equation}

An admissible surface has a \emph{constant ratio $a$ of isotropic principal curvatures}, or is an \emph{isotropic CRPC surface}, if $K\neq 0$ and $\kappa_1/\kappa_2=a$ or $\kappa_2/\kappa_1=a$ at each point.
Equivalently, $H^2/K=(a+1)^2/(4a)=\mathrm{const}$. For $K<0$, this is equivalent to~\eqref{eq-isotropic-crpc} with $t=(a+1)/\sqrt{|a|}$.
For instance, an \emph{isotropic minimal surface}, i.e., the one with $H=0$ everywhere, has ratio $a=-1$. For $a=1$, we get a unique up to general isotropic similarity CRPC surface $2z = x^2+y^2$, also known as the \emph{isotropic unit sphere} \cite[Section~62, p.~402]{strubecker:1942}. 


The asymptotic directions of a surface are the same in Euclidean and isotropic geometry. For an isotropic CRPC surface with $a<0$, the asymptotic curves meet under the constant isotropic angle $\gamma$, given by $\cot^{2}(\gamma/2)=|a|$ \cite[Section~2.1]{Yorov-Pottmann-Skopenkov-23}. 

\begin{figure}[htb]
\hfill
\begin{overpic}[width=0.2\textwidth]{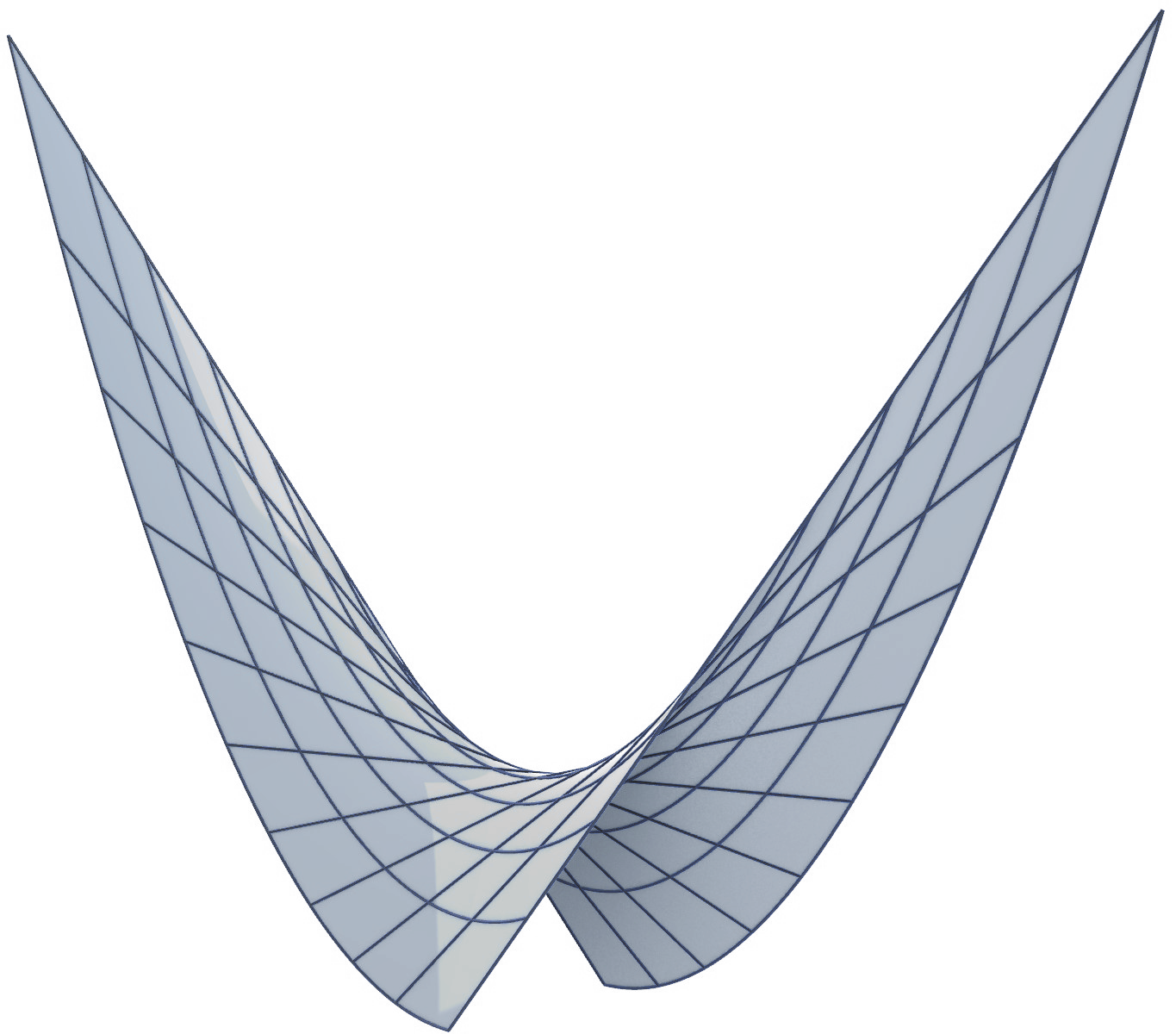}
\end{overpic}
\hfill
\begin{overpic}[width=0.2\textwidth]{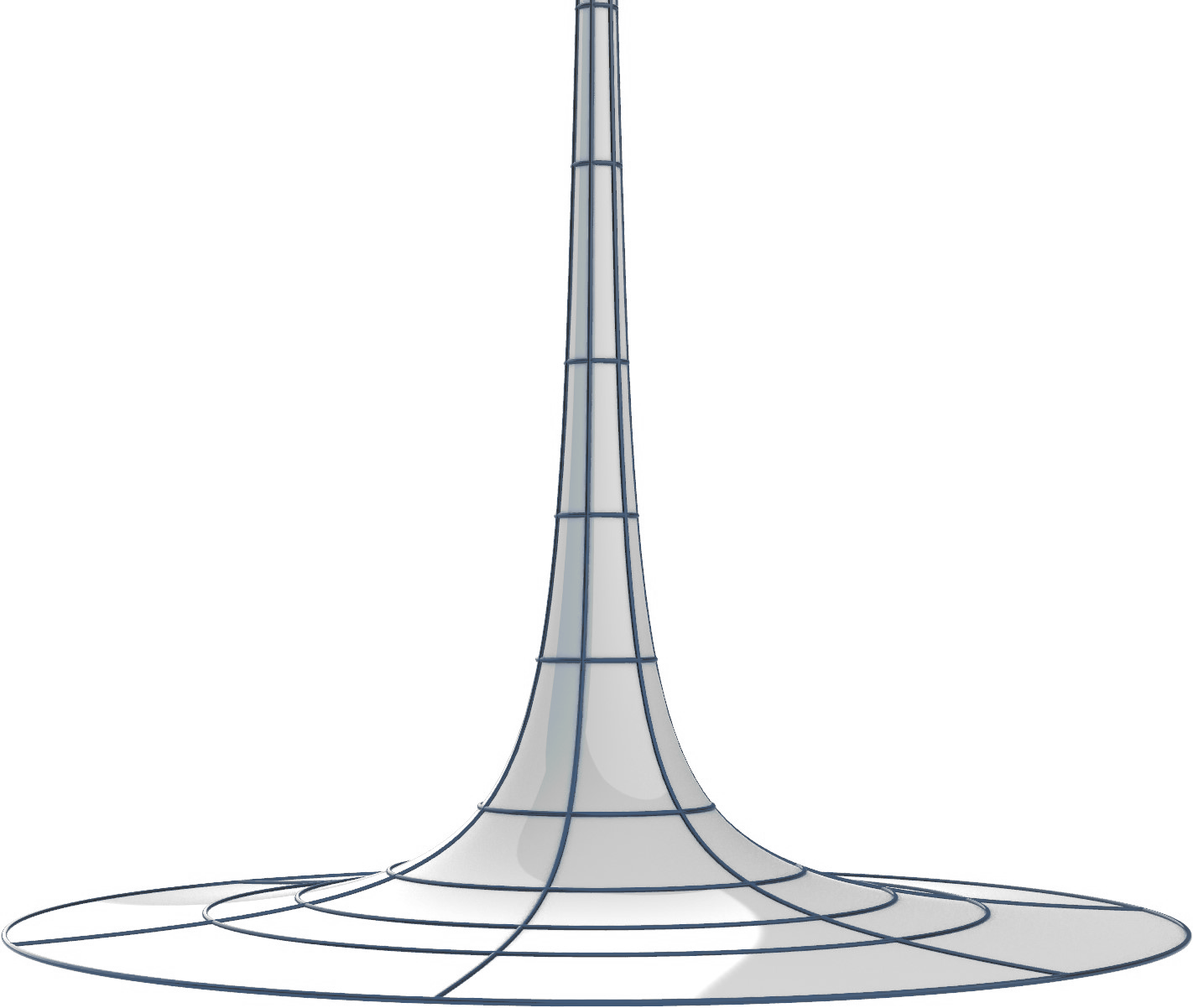}
\end{overpic}
\hfill
\begin{overpic}[width=0.15\textwidth]{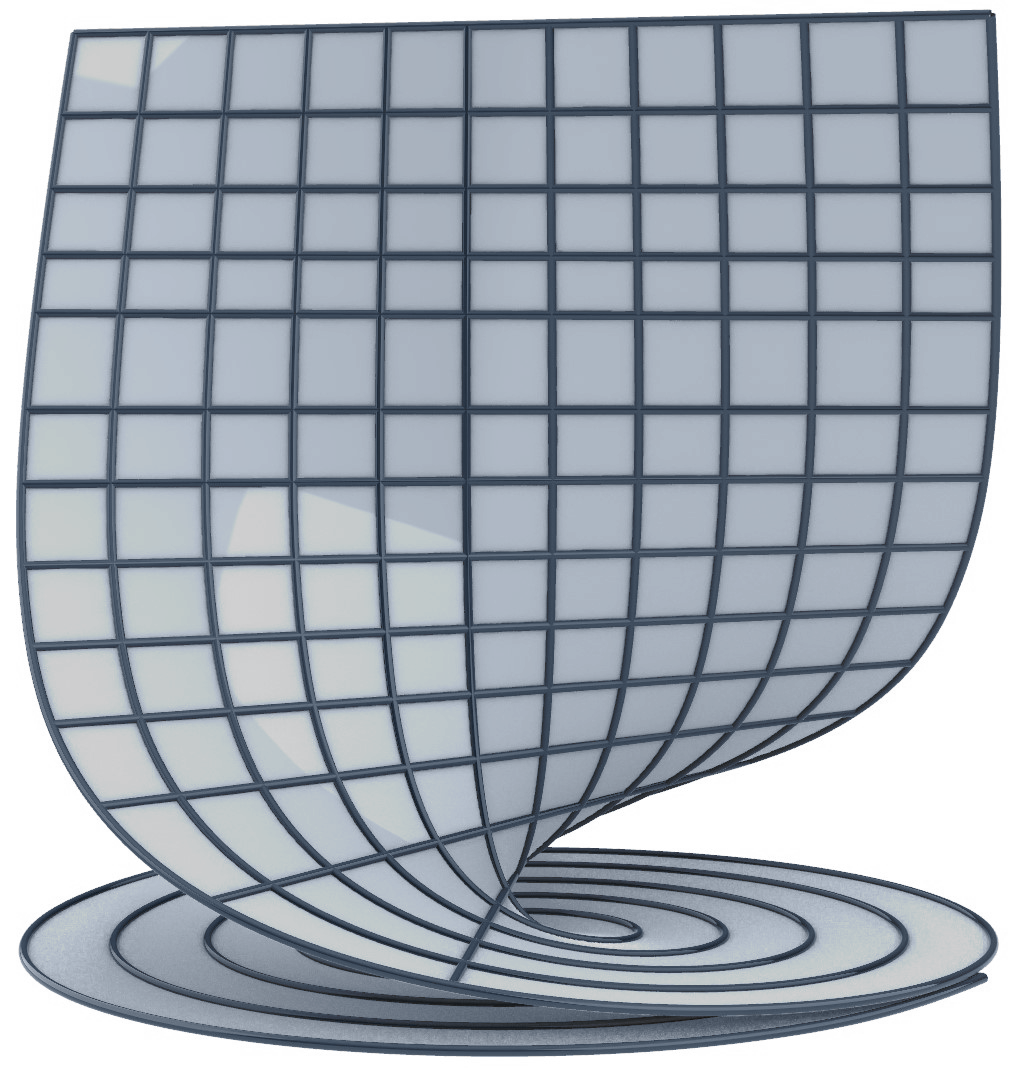}
\end{overpic}
\hfill
\begin{overpic}[width=0.18\textwidth]{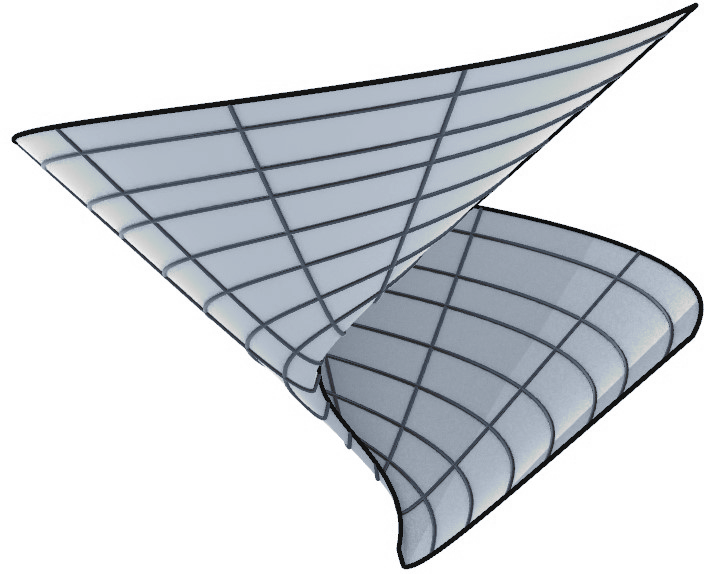}
\end{overpic}
\hfill
\begin{overpic}[width=0.23\textwidth]{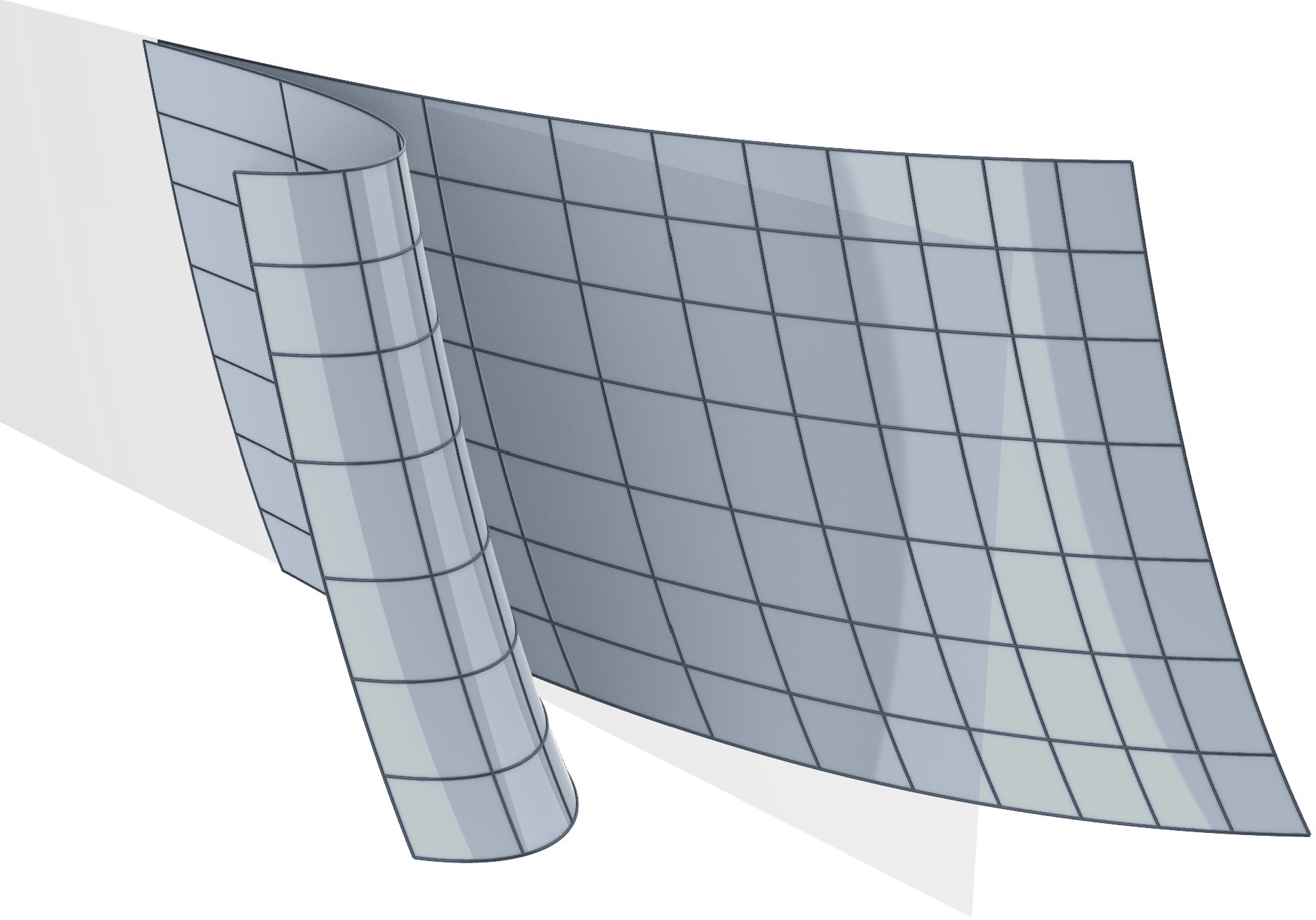}
\end{overpic}
    \caption{Isotropic CRPC surfaces (from left to right): a hyperbolic paraboloid (Example~\ref{ex:eq-paraboloid}), a rotational surface (Example~\ref{ex:rot-isotropic}), 
    a ruled surface (Example~\ref{ex:half-helicoid}), 
     a helical surface (Example~\ref{ex:part-helicatenoid}), a translational surface (Example~\ref{ex:translational}).  
     } 
    \label{fig:2}
\end{figure}

\subsection{Isotropic CRPC surfaces}

Let us list the known explicit families of isotropic CRPC surfaces. See Fig.~\ref{fig:2}. 

The simplest family is formed by paraboloids with $z$-parallel axes, which are always isotropic CRPC surfaces 
\cite[Example~1]{Yorov-Pottmann-Skopenkov-23}.

\begin{example}\label{ex:eq-paraboloid} The hyperbolic paraboloid $z=2xy$, where $x^2+y^2\le 1$, is contained in the analytic family of paraboloids
with the same boundary and the ratio $a$ of isotropic principal curvatures:
\begin{equation}\label{eq-paraboloid}
z=2xy+\frac{a+1}{a-1}\left(x^2+y^2-1\right), \qquad a\in [-1;0),
\end{equation}

In other words, the harmonic function $f^{(0)}(x,y)=2xy$ in the unit circle is contained in the real analytic family of functions having the same boundary values and satisfying~\eqref{eq-isotropic-crpc}:
\begin{equation}\label{eq-deg2}
f(x,y,t)=2xy+\frac{t}{\sqrt{t^2+4}}(x^2+y^2-1)
\end{equation}
\end{example}

Already, the case when $f^{(0)}(x,y)$ is a cubic harmonic polynomial seems not solvable in a closed form. However, this case demonstrates that flat points usually lead to singularities and should thus be excluded. 

\begin{example} \label{ex-no-xxxt} The harmonic function $f^{(0)}(x,y)=x^3-3xy^2$ has a flat point at the origin and is \emph{not} contained even in a $C^4$ family of functions $f(x,y,t)$ satisfying~\eqref{eq-isotropic-crpc}. Indeed, differentiating~\eqref{eq-isotropic-crpc} with respect to $t$ and setting $t=0$, we get
\begin{equation*}
\left.f_{xxt}+f_{yyt}\right|_{t=0}=6\sqrt{x^2+y^2}.
\end{equation*}
Here, the right side, and hence the left one, is not differentiable at the origin. See also Example~\ref{ex-no-xxtt} below. 

More generally, if a minimal surface $\Phi^0$ has a flat point, where $K=K_x=K_y=0$ but $K_{xx}K_{yy}-K_{xy}^2\ne 0$, then $\Phi^0$ is \emph{not} contained in a $C^4$ family of surfaces $\Phi^s$ with the ratio of the isotropic principal curvatures $s-1$.
Again, differentiating~\eqref{eq-isotropic-crpc}, we get
\begin{equation*}
\left.2H_t\right|_{t=0}=\sqrt{-K},
\end{equation*}
where the right side, hence the left one, is not differentiable
at the flat point. This argument also works in \emph{Euclidean} geometry. However, it applies to ``non-multiple'' flat points only, as we show next. 

\end{example}

\begin{remark}\label{rem-nonvanishing-det}
The (isotropic or Euclidean) Gaussian curvature is strictly negative at any non-flat point of a (isotropic or Euclidean) minimal surface.
\end{remark}

\begin{example} The constant $f^0(x,y)=0$ is contained in a real analytic family of nonlinear (for $t\ne 0$) functions $f(x,y,t)$ satisfying~\eqref{eq-isotropic-crpc}. Indeed, if some nonlinear function $f(x,y,t)$ satisfies~\eqref{eq-isotropic-crpc}, then 
$c(t)f(x,y,t)$ also does, for any real analytic function $c(t)$. In particular, if $c(0)=0$, then
the latter family contains the constant $0$ for $t=0$. 
\end{example}

The known closed-form families that we present now are formed by natural classes of surfaces such as rotational, ruled, helical, and translational ones. Those families contain all the isotropic CRPC surfaces in the corresponding classes, besides a few exceptions in the case of helical and translational surfaces \cite[Proposition~3, Theorems~17, 22, and~23]{Yorov-Pottmann-Skopenkov-23}.


\begin{example}\label{ex:rot-isotropic} (Cf.~Example~\ref{ex-catenoid}.) In the cylindrical coordinate system $(r,\phi,z)$, take the rotational isotropic minimal surface
$
z=\log r, 
$
where
$
r\ge 1
$
(\emph{isotropic catenoid}).
Then 
it is contained (as the limiting 
case $a\to -1$) in the analytic 
family of unbounded rotational surfaces with the same boundary and the ratio $a$ of isotropic principal curvatures:
$$
z=\frac{r^{a+1}-1}{a+1},
\qquad r\ge 1,\quad a\in (-1;0).
$$
For $0<a\le 1$, there is a rotational $C^2$ surface with the same boundary and the ratio $a$ of isotropic principal curvatures, depending on an additional parameter $c>0$:
$$
z=c\,r^{1+1/a}-c,
\qquad 0\le r\le 1,\quad a\in (0;1],\quad c>0.
$$
Thus, for each $a\in (0;1]$, Plateau's problem for the unit circle 
has uncountably many solutions. For $a\ne 1$, the surface has a flat point on the $z$-axis and is not analytic at the point unless $a=1,1/3,1/5,\dots$.
\end{example}


\begin{example}\label{ex:half-helicoid} The 
helicoid 
$y=x\tan z$, 
where $z \in (0, \pi /2)$ and $x > 0$, 
is contained (as the limiting case $a\to -1$) in the analytic family of 
ruled surfaces with 
the ratio  $a$ of isotropic principal curvatures:
$$y=x\tan \left(\frac{\sqrt{|a|}}{1+a}\log \left((a+1)z+1\right)\right),
\qquad x > 0, \quad z \in (0; \pi /2), \quad  a\in (-1;0).
$$
%
\end{example}

\begin{example}\label{ex:part-helicatenoid} The 
isotropic minimal surface
$r_c(u,v) = [u\cos v, u\sin v, c\,\log u+v]$, 
where $u>0$ and $v \in ( -\frac{\pi}{2}, \frac{\pi}{2})$ (\emph{isotropic helicatenoid}), 
is contained (as the limiting case $a\to -1$) in the analytic family of unbounded helical surfaces with the 
ratio $a$ of isotropic principal curvatures, where $b = \sqrt{1+c^2}-c$:
        \[
r_{a, b}(u,v) =  
\begin{bmatrix}
{\sqrt{u^{-2a} + b^{2} u^{2}}} \, \cos v /{\sqrt{b^{2} + 1}}\\
{\sqrt{u^{-2a} + b^{2} u^{2}}} \, \sin v /{\sqrt{b^{2} + 1}}\\
v + \arctan\left(\, bu^{a+1}\right) - \arctan b + \frac{b^2 \left(u^{a+1}-1\right) + a^{2}\!\left(u^{-a-1} - 1\right)}{\left(a^{2} - 1\right)b} 
\end{bmatrix}
, \quad u > 0, \, v \in \left( -\frac{\pi}{2};\frac{\pi}{2}\right),
         \,  a\in (-1;0).
\]
\end{example}


The surface $r_{a,b}(u,v)$ has indeed the ratio $a$ of isotropic principal curvatures because it is obtained from the surface in \cite[Eq.~(10)]{Yorov-Pottmann-Skopenkov-23} 
by 
the change of variable $u\mapsto \arctan bu^{a+1}$
and the isotropic similarity 
$$(x,y,z) \mapsto \left(\frac{b^{\frac{a}{a+1}}}{\sqrt{b^2+1}}\, x, \,\frac{b^{\frac{a}{a+1}}}{\sqrt{b^2+1}}\, y, \, z  - \arctan b -\frac{b^2 +a^2}{(a^2-1)b} \right). $$ 
The helicatenoid $r_c(u,v)$ is a limiting case because 
        $
        \lim_{a \to -1} r_{a,b}(u,v) = 
        r_c(u,v).
        $ 

Notice that the families in the last two examples are completely different, although they contain (different parts of) the same helicoid (for $c=0$): they consist of ruled and helical surfaces, respectively. 
The latter family is paradoxical: the isotropic minimal surfaces for different values of $c$ are 
different, but all arise as limiting cases of the same (up to isotropic similarity) family of isotropic CRPC surfaces.


\begin{example}\label{ex:translational}
The 
translational isotropic minimal surface
$y =\log \frac{z}{\sin x}$, where $x \in (0; \pi)$ and $z \in [1; e]$ (\emph{first isotropic Scherk surface}), 
is contained (as the limiting case $b\to 0$) in the analytic 
family of translational surfaces with the 
ratio of isotropic principal curvatures~$(b+1)/(b-1)$:
\begin{align}
r_b(u,v) = 
\label{eq-isotropic+nonisotropic}
\begin{bmatrix}
          v+ b\cos v\\
            b\sin v+(b^2-1)\log\left\lvert b-\sin v\right\rvert+(1-b^2)u \\
           \exp u
         \end{bmatrix},
        \qquad 
        u \in [0; 1],
        \quad
        v\in \left(0; \pi\right), 
        \quad 
         b\in \left(-1; 0\right),
\end{align}
\end{example}


    Here, the surface $r_b(u,v)$ has the ratio of isotropic principal curvatures ~$(b+1)/(b-1)$ by \cite[Theorem~23]{Yorov-Pottmann-Skopenkov-23}. 
    The first isotropic Scherk surface is a limiting case because it can be
    parametrized as $
        r(u,v) =
        [        v,
        u - \log\sin v,
        \exp u]
        $
        and
        $
        \lim_{b \to 0} r_b(u,v) 
         = r(u,v).
        $

There is also a \emph{second isotropic Scherk surface} $
        r(u,v) =
        [   u+v,
        \log|\cos u| - \log|\cos v|,
        u]
        $
\cite[Eq.~(22)]{Yorov-Pottmann-Skopenkov-23}.

\begin{problem} \label{pr-transaltional} Find a closed-form family of translational Euclidean (isotropic) CRPC surfaces containing Euclidean (isotropic) Scherk's second surface.
\end{problem}


\section{Isotropic CRPC surfaces}\label{sec-proofs-smooth}

In this section, we prove analogs of Theorem~\ref{thm-euclidean} and Corollary~\ref{thm-plateau-euc} in isotropic geometry. These analogs are obtained 
by replacing 
``minimal surface'' with ``isotropic minimal surface'', and ``principal curvatures'' with ``isotropic principal curvatures''. We prove them in the following equivalent form.

\begin{theorem}\label{thm-isotropic-equivalent}
Let $f^{(0)}(x,y)$ be a harmonic function in a Jordan domain with an analytic boundary that
extends to a real analytic function in the closed domain with a 
nonvanishing Hessian (even on the boundary). Then there is a unique (up to restriction) real analytic family of functions $f(x,y,t)$, where $t$ changes in a vicinity of $0$, having the same boundary values and satisfying $f(x,y,0)=f^{(0)}(x,y)$ and~\eqref{eq-isotropic-crpc}.
\end{theorem}



%
%

\begin{corollary}
\label{thm-plateau-iso} 
     If an analytic Jordan curve satisfies the assumptions of Corollary~\ref{thm-plateau-euc},
     then for all sufficiently small~$t$, it can be spanned by the graph of a function 
     that satisfies~\eqref{eq-isotropic-crpc}.
\end{corollary}

In what follows, we fix a Jordan domain $\Omega$ with an analytic boundary and a parameter $\varepsilon>0$. Denote $\delta_{mn} = 1$ if $m=n$ and $0$ otherwise. Any empty sum is set to be $0$ by definition. 

We start with two lemmas giving a recursion for the coefficients of series~\eqref{eq-expansion}.


\begin{lemma}\label{der-mean}
Two real analytic functions $H, K\colon(-\varepsilon; \varepsilon)\to\mathbb{R}$ with $K(0)< 0$ satisfy $4H^2 = -t^2K$ for all $|t|<\varepsilon$
%
%
if and only if $H(0) = 0$ and
\begin{equation}\label{H_m}
    H^{(m)} = \pm\frac{-mK^{(m-1)}-\frac{4}{m+1}\sum_{r=2}^{m-1}\binom{m+1}{r}
H^{(r)}H^{(m-r+1)}}{2^{2-\delta_{m1}}\sqrt{-K^{(0)}}}
\qquad\text{for each }m=1,2,\dots,
\end{equation}
where $H^{(i)}$ and $K^{(i)}$ are the $i$-th derivatives at $t=0$,
and the same choice of sign in $\pm$ applies for all $m$.
\end{lemma}

\begin{proof}
Assume that $4H^2=-t^2K$. Then $H(0)$ vanishes. Differentiating both sides $m + 1$ times, using the Leibniz rule,  and setting $t=0$, we get
\begin{equation}\label{eq-der-mean-proof}
   4\sum_{r=0}^{m+1}\binom{m+1}{r}H^{(r)}H^{(m-r+1)} = -\sum_{r=0}^{m+1}\binom{m+1}{r}
   2\delta_{2r}
   K^{(m-r+1)}.
\end{equation}
This is equivalent to 
$2H^{(1)} =\pm\sqrt{-K^{(0)}}$, 
when $m=1$, and to
\begin{equation*}
   8H^{(0)}H^{(m+1)}+ 8(m+1)H^{(1)}H^{(m)}+4\sum_{r=2}^{m-1}\binom{m+1}{r}H^{(r)}H^{(m-r + 1)} = 
   - m(m+1)K^{(m-1)},
\end{equation*}
when $m>1$. Since $H^{(0)}=0$ and $K^{(0)}< 0$,  we get~\eqref{H_m} in both cases.
Conversely, \eqref{eq-der-mean-proof} and $H^{(0)} = 0$ 
imply 
the coincidence of Taylor coefficients of 
$4H^2$ and $-t^2K$. By the analyticity, \eqref{H_m} and $H(0) = 0$ 
imply 
$4H^2=-t^2K$.
\end{proof}

Recall that \emph{$C^k$-convergence} means uniform convergence of a sequence of functions together with all the 
partial derivatives up to order $k$. 

\begin{lemma} \label{l-poisson} 
Assume that for a sequence of $C^2$ functions $f^{(m)}\colon\Omega \to\mathbb{R}$, we have $f^{(0)}_{xx}f^{(0)}_{yy}-\left(f^{(0)}_{xy}\right)^2<0$, and series~\eqref{eq-expansion} $C^2$-converges for $|t|<\varepsilon$. Then the sum
%
satisfies~\eqref{eq-isotropic-crpc} if and only if for each $m\ge 0$ 
we have
\begin{align}\label{eq-poisson}
\textstyle f^{(m)}_{xx}+f^{(m)}_{yy} &= \notag\\
 & \textstyle  \frac{
m\sum_{r=0}^{m-1}\binom{m-1}{r}\left(f^{(r)}_{xy}f^{(m-r-1)}_{xy}
-f^{(r)}_{xx}f^{(m-r-1)}_{yy}\right)
-\frac{1}{m+1}\sum_{r=2}^{m-1}\binom{m+1}{r}
\left(f^{(r)}_{xx}+f^{(r)}_{yy}\right)\left(f^{(m-r+1)}_{xx}+f^{(m-r+1)}_{yy}\right)
}{2^{1-\delta_{m1}}\sqrt{\left(f^{(0)}_{xy}\right)^2-f^{(0)}_{xx}f^{(0)}_{yy}}}. 
\end{align}
\end{lemma}

\begin{proof} 
Since series~\eqref{eq-expansion} $C^2$-converges, 
it follows that $f_{xx}=\sum_{m=0}^\infty f_{xx}^{(m)}t^m/m!$, $f_{xy}=\sum_{m=0}^\infty f_{xy}^{(m)}t^m/m!$, $f_{yy}=\sum_{m=0}^\infty f_{yy}^{(m)}t^m/m!$, and hence $H = \left(f_{xx} + f_{yy}\right)/2$ and $K= f_{xx}f_{yy}-f_{xy}^2$ are real analytic in $t$ 
for each $(x,y)\in\Omega$.
Moreover, ~\eqref{eq-isotropic-crpc} is equivalent to $2H=t\sqrt{-K}$.
By Lemma~\ref{der-mean} and the Leibniz rule, \eqref{eq-isotropic-crpc} is equivalent to~\eqref{eq-poisson}. 
Note that the sign in $\pm$ in~\eqref{H_m} is uniquely determined to be `$+$' because $2H^{(1)}=\sqrt{-K^{(0)}}$ by~\eqref{eq-isotropic-crpc}.
\end{proof}

We obtain the following direct corollary. 

\begin{corollary}[Uniqueness] \label{cor-uniqueness} 
Under the assumptions of Theorem~\ref{thm-isotropic-equivalent}, there is no more than one real analytic family $f\colon\overline\Omega\times [-\varepsilon;\varepsilon]\to\mathbb{R}$
satisfying~\eqref{eq-isotropic-crpc} such that $f(x,y,t)=f^{(0)}(x,y)$ for $t=0$ or $(x,y)\in\partial\Omega$.
\end{corollary}

\begin{proof} 
 Let $f^{(m)}$ be the $m$-th derivative of $f(x,y,t)$ with respect to $t$ evaluated at $t=0$. It suffices to prove that $f^{(m)}$ is uniquely determined by $f^{(0)}$ for each $m\ge 1$. 
 We prove it by induction on $m$. 
 Assume that 
 $f^{(0)},\dots,f^{(m-1)}$ have already been determined. Since $f(x,y,t)=f^{(0)}(x,y)$ on the boundary $\partial\Omega$, it follows that $f^{(m)}(x,y)=0$ on $\partial\Omega$. By the assumptions of Theorem~\ref{thm-isotropic-equivalent}, $f^{(0)}(x,y)$ is harmonic with a nonvanishing Hessian; hence, $f^{(0)}_{xx}f^{(0)}_{yy}-\left(f^{(0)}_{xy}\right)^2<0$ by Remark~\ref{rem-nonvanishing-det}.
 By compactness, series~\eqref{eq-expansion} $C^2$-converges for sufficiently small $|t|$, because the Taylor series of the derivatives of $f(x,y,t)$ with respect to $x$ and $y$ converge uniformly in a neighborhood of each point of $\overline\Omega\times\{0\}$.
 Thus, by Lemma~\ref{l-poisson}, we get~\eqref{eq-poisson}. This is a Poisson equation with respect to $f^{(m)}$ because the right side depends only on $f^{(0)},\dots,f^{(m-1)}$. Since the Poisson equation with the Dirichlet boundary condition has no more than one solution, $f^{(m)}$ is uniquely determined. By induction, the corollary follows.
\end{proof}

For the existence of the family $f(x,y,t)$, we need the $C^k$-convergence of series~\eqref{eq-expansion}, which is equivalent to the convergence in the \emph{$C^k$ norm}
\begin{equation}\label{eq-norm}
\left\|f\right\|_{C^{k}}:= \sum_{n=0}^{k}\;\sup_{\substack{(x,y) \in \Omega\\ 0\le i\le n}}\;\left|\frac{\partial^{n}f}{\partial x^{i}\partial y^{n-i}}\right|.
\end{equation}

We are going to use the following well-known version of the Weierstrass M-test.

\begin{lemma}\label{convergence} Let $t\in\mathbb{R}$. If a sequence of $C^k$ functions $f^{(m)}\colon\Omega\to\mathbb{R}$ and a sequence $a^{(m)}$ of real numbers
satisfy $\left\|f^{(m)}\right\|_{C^{k}} \le a^{(m)}$
for $m\ge 0$ and $\sum_{m=0}^{\infty}a^{(m)}|t|^m/m!$ converges, then series~\eqref{eq-expansion} $C^k$-converges. 
\end{lemma}

So, we need a bound for the growth of 
$\left\|f^{(m)}\right\|_{C^{k}}$ as $m\to\infty$. 
The following simple example shows that we cannot get it just from~\eqref{eq-poisson} viewed as the Poisson equation; see \cite{Urbano-2017,Di-Fazio} for details and generalizations. 

\begin{example}[Folklore] The function $f(x,y)=xy\log(x^2+y^2)$ (extended by $0$ at the origin) has bounded Laplacian $f_{xx}+f_{yy}$ but unbounded derivative $f_{xy}$. This shows that just a bound for $\|f_{xx}+f_{yy}\|_{C^0}$ does not provide a bound for $\|f\|_{C^2}$.
\end{example}

The way out is to consider the H\"older norms instead, for which Schauder's estimates do the job.



\begin{definition}[H\"older norm] \label{def-holder}
A function $f\colon\Omega\to\mathbb{R}$
is \emph{H\"older} with the exponent $1/2$, if $|f(u)-f(v)|\le C|u-v|^{1/2}$ for some constant $C>0$ and all $u,v\in \Omega$. Its \emph{H\"older semi-norm} is 
$$[f]_{\Omega} := \sup_{u,v\in \Omega: u\ne v}\frac{|f(u)-f(v)|}{|u-v|^{1/2}}.$$
A function $f\colon\overline\Omega\to\mathbb{R}$ is a \emph{$C^{k+1/2}$ function}, if all the partial derivatives of order $\le k$ exist in $\Omega$, have continuous extensions to the closure 
$\overline\Omega$, and are H\"older with the exponent $1/2$. 
Its \emph{H\"older norm} is
\begin{align*}
  \|f\|_{C^{k +1/2}}:= \|f\|_{C^{k}}+
   \sup_{\;0\le i\le k}\ \left[\frac{\partial^{k}f}{\partial x^i\partial y^{k-i}}\right]_{\Omega},
\end{align*}
where we use notation~\eqref{eq-norm}.
A \emph{$C^\infty$ function} $f\colon\overline{\Omega}\to\mathbb{R}$ is a continuous function for which all the partial derivatives of all orders exist in $\Omega$ and have continuous extensions to $\overline{\Omega}$.




\end{definition}

\begin{theorem}[Schauder's estimates] \label{thm-schauder} \textup{(See \cite[Theorems~1.1 and~1.2]{Koenig-11}; cf.~\cite[Theorem~7.3 and Remark~2 after it]{Agmon-etal-59} and \cite[end~of~\S6.3]{Gilbarg-Trudinger-83})}
For each $C^{\infty}$ function $g\colon\overline\Omega\to \mathbb{R}$, the equation $f_{xx}+f_{yy}=g$ has a unique $C^{\infty}$ solution $f\colon\overline\Omega\to\mathbb{R}$ vanishing on $\partial\Omega$.
There is a constant $C(\Omega)>0$ depending only on $\Omega$ 
such that 
\begin{equation*}
  \|f\|_{C^{2+1/2}}\le C(\Omega) \|g\|_{C^{1/2}}.
\end{equation*}
\end{theorem}

We use this result to obtain a recursive bound for $\|f^{(m)}\|_{C^{2+1/2}}$. Hereafter, abbreviate
$$\|f\|:=\|f\|_{C^{2+1/2}}.$$

\begin{lemma} \label{l-one-step} 
Starting with a harmonic function $f^{(0)}\colon\Omega\to\mathbb{R}$,
which extends to a real analytic function in $\overline\Omega$ and has a nonvanishing Hessian (even on $\partial\Omega$), 
define a sequence of $C^\infty$ functions $f^{(m)}\colon\overline\Omega\to\mathbb{R}$, where $m\ge 1$, inductively as unique solutions of~\eqref{eq-poisson} vanishing on $\partial\Omega$. Denote $K^{(0)}=f^{(0)}_{xx}f^{(0)}_{yy}-\left(f^{(0)}_{xy}\right)^2$. 
Then 
there is $C(\Omega)>0$ such that for each $m\ge 1$ we have
\begin{equation*}
 \textstyle \left\|f^{(m)}\right\|
  \le C(\Omega)
  \left\|\frac{1}{\sqrt{-K^{(0)}}}\right\|
  \left(
  m\sum_{r=0}^{m-1}\binom{m-1}{r}
  \left\|f^{(r)}\right\|
  \left\|f^{(m-r-1)}\right\|
  +\frac{1}{m+1}\sum_{r=2}^{m-1}\binom{m+1}{r}
  \left\|f^{(r)}\right\|
  \left\|f^{(m-r+1)}\right\|
  \right).
\end{equation*}
\end{lemma}

\begin{proof} 
    Since $f^{(0)}$ is harmonic 
    and has a nonvanishing Hessian (even on $\partial\Omega$), we get 
    $K^{(0)}<0$ in $\overline{\Omega}$ by Remark~\ref{rem-nonvanishing-det}.
    Denote $\left\|f\right\|{}':=\left\|f\right\|_{C^{1/2}}$ so that $\left\|f_{xx}\right\|{}',\left\|f_{xy}\right\|{}',\left\|f_{yy}\right\|{}'\le \left\|f\right\|$.  
    Applying Schauder's estimates (Theorem~\ref{thm-schauder}) to~\eqref{eq-poisson} and using the submultiplicativity of the H\"older norm 
    $\left\|fg\right\|'\le C(\Omega)\left\|f\right\|' \left\|g\right\|'$,
    we get, for some $C_1(\Omega)$, $C_2(\Omega)$, and $C(\Omega)$,
\begin{align*}
 \left\|f^{(m)}\right\| &\le  C_1(\Omega)\left\|\frac{m\sum_{r=0}^{m-1}\binom{m-1}{r}\left(f^{(r)}_{xy}f^{(m-r-1)}_{xy}
-f^{(r)}_{xx}f^{(m-r-1)}_{yy}\right)}{2^{1-\delta_{m1}}\sqrt{\left(f^{(0)}_{xy}\right)^2-f^{(0)}_{xx}f^{(0)}_{yy}}}
\right.
\\ &\quad-
\left.
\frac{\frac{1}{m+1}\sum_{r=2}^{m-1}\binom{m+1}{r}
\left(f^{(r)}_{xx}+f^{(r)}_{yy}\right) \left(f^{(m-r+1)}_{xx}+f^{(m-r+1)}_{yy}\right)
}{2^{1-\delta_{m1}}\sqrt{\left(f^{(0)}_{xy}\right)^2-f^{(0)}_{xx}f^{(0)}_{yy}}}\right\|'\\
&\le \textstyle C_2(\Omega) \left\|\frac{1}{\sqrt{-K^{(0)}}}\right\|'\left( m\sum_{r=0}^{m-1}\binom{m-1}{r}\left(
\left\|f^{(r)}_{xy}\right\|' \left\|f^{(m-r-1)}_{xy}\right\|'+
\left\|f^{(r)}_{xx}\right\|' \left\|f^{(m-r-1)}_{yy}\right\|'
\right)
\right.
\\ &\quad+ \textstyle\left.\frac{1}{m+1}\sum_{r=2}^{m-1}\binom{m+1}{r}
\left(\left\|f^{(r)}_{xx}\right\|'+\left\|f^{(r)}_{yy}\right\|'\right)\left(\left\|f^{(m-r+1)}_{xx}\right\|'+\left\|f^{(m-r+1)}_{yy}\right\|'\right)\right)\\
 &\le \textstyle C(\Omega) \left\|\frac{1}{\sqrt{-K^{(0)}}}\right\|\left(m\sum_{r=0}^{m-1}\binom{m-1}{r} \left\|f^{(r)}\right\|\left\|f^{(m-r-1)}\right\| + \frac{1}{m+1}\sum_{r=2}^{m-1}\binom{m+1}{r}
  \left\|f^{(r)}\right\|
  \left\|f^{(m-r+1)}\right\|\right).\\[-0.8cm]
\end{align*}
\end{proof}

Now we majorize $\left\|f^{(m)}\right\|$ by the Taylor coefficients of an explicit real analytic function. Abbreviate $M:=\left\|f^{(0)}\right\|$ and $N:=C(\Omega)\left\|{1}/{\sqrt{-K^{(0)}}}\right\|.$

\begin{lemma} \label{l-reccurence} Let $M,N>0$. 
Define a sequence $a^{(m)}=a^{(m)}(M,N)$ 
recursively by $a^{(0)}=M$ and
\begin{equation*}
  a^{(m)} :=
  Nm\sum_{r=0}^{m-1}\binom{m-1}{r}
  a^{(r)}
  a^{(m-r-1)}
  +\frac{N}{m+1}\sum_{r=2}^{m-1}\binom{m+1}{r}
  a^{(r)}a^{(m-r+1)}
\end{equation*}
for each $m=1,2,\dots$. Then $\sum_{m=0}^{\infty}a^{(m)}t^m/m!$ converges
in a neighborhood of $0$, and the sum 
equals
\begin{equation}\label{eq-l-reccurence}
a(t)=\frac{M+t(M^2N+1/(2N))-t\sqrt{1/(4N^2)-t(2M^3N+M/N)-t^2M^4N^2}}{t^2+1}.
\end{equation}
\end{lemma}


\begin{proof}
Denote $b_{m}:=a^{(m)}/m!$ for each $m\ge 0$ and introduce formal power series $a(t) =  \sum_{m=0}^{\infty}a^{(m)}t^m/m!$. Then 
 \begin{align*}
    b_{m} &= N \sum_{r=0}^{m-1}b_{r}b_{m-r-1}+ N \sum_{r=2}^{m-1}b_{r}b_{m-r+1} \qquad\text{for } m\ge 1.\\ 
    \intertext{In particular, $b_0 = M$ and $b_1 = M^2N$. Hence, by the properties of formal power series, we get}
    a(t) &= b_{0} + N t\sum_{m=1}^{\infty}\sum_{r=0}^{m-1}b_{r}b_{m-r-1}t^{m-1} + \frac{N}{t}\sum_{m=3}^{\infty}\sum_{r=2}^{m-1}b_{r}b_{m-r+1}t^{m+1}= M + Nta(t)^2 + \frac{N}{t}\left(a(t)-M^2Nt-M\right)^2.
 \end{align*}
Solving the resulting quadratic equation with respect to $a(t)$, we get~\eqref{eq-l-reccurence}. 
We have a minus sign in front of the square root because $a'(0)=b_1=M^2N$.
Since the right side of~\eqref{eq-l-reccurence} is analytic in a neighborhood of $t=0$, the series $\sum_{m=0}^{\infty}a^{(m)}t^m/m!$ converges in a neighborhood of $0$ and the sum equals~\eqref{eq-l-reccurence}. 
\end{proof}

By induction on $m$, Lemmas~\ref{l-poisson}, \ref{convergence}, \ref{l-one-step}, and~\ref{l-reccurence} imply the following. 

\begin{corollary}\label{fmlessiso} Under the assumptions and notation of Lemmas~\ref{l-one-step} and~\ref{l-reccurence}, we have
\begin{equation*}
  \left\|f^{(m)}\right\|\le a^{(m)}\left(
  \left\|f^{(0)}\right\|,
  C(\Omega)\left\|\frac{1}{\sqrt{-K^{(0)}}}\right\|\right) \qquad\text{for each $m\ge 0$.}
\end{equation*}
Hence, series~\eqref{eq-expansion} converges in the $C^{2+1/2}$ norm for 
sufficiently small $|t|$, and the sum satisfies~\eqref{eq-isotropic-crpc}.
\end{corollary}


\begin{remark}
This corollary and its proof remain true if we only assume that $\partial\Omega$ and $f^{(0)}\colon\overline\Omega\to\mathbb{R}$
are of class $C^{2+1/2}$ (instead of the analyticity), so that we get a $C^{2+1/2}$ solution to~\eqref{eq-isotropic-crpc}. 
\end{remark}

The analyticity of the sum is guaranteed by Friedman's theorem, 
which we state in a particular case.

\begin{theorem}[Friedman] 
(See \cite[Theorem~5, remark after it, and the bottom of p.~44]{Friedman-58}; cf.~\cite[Chapter~4, p.~63]{petrowsky-1996-i}.)
\label{th-petrowski-short}
Consider the partial differential 
equation
\begin{equation}\label{eq-petrowski}
F(f_{xx},f_{xy},f_{yy},f_{x},f_{y})=0,
\end{equation}
where $F$ is a complex analytic function in a domain $D\subset\mathbb{C}^5$
such that for all $(\xi_1,\xi_2)\in\mathbb{R}^2\setminus\{(0,0)\}$,
\begin{equation}\label{eq-l-petrowski}
\frac{\partial{F}}{\partial f_{xx}}\xi_1^2+
\frac{\partial{F}}{\partial f_{xy}}\xi_1\xi_2+
\frac{\partial{F}}{\partial f_{yy}}\xi_2^2\ne 0.
\end{equation}
Let $f\colon \overline\Omega\to \mathbb{R}$ be a $C^{2+1/2}$ function 
such that for all $(x,y)\in \overline\Omega$, we have $(f_{xx},f_{xy},f_{yy},f_{x},f_{y})\subset D$ and~\eqref{eq-petrowski} holds. Then $f$ is real analytic in the closed domain.
\end{theorem}

\begin{corollary}[Analyticity of isotropic CPRC surfaces] \label{cor-crpr-analyticity-isotropic}
    Any $C^{2+1/2}$ function $f\colon \overline\Omega\to \mathbb{R}$ satisfying equation~\eqref{eq-isotropic-crpc} for some $t\in\mathbb{R}$ and the inequality $f_{xx}f_{yy}-f^2_{xy}<0$ is real analytic in the closed domain.
\end{corollary}

\begin{proof}  By compactness, there exists $\varepsilon>0$ such that 
$f_{xx}f_{yy}-\left(f_{xy}\right)^2 < -\varepsilon^2$.
Introduce the function $$F(A,B,C)=A+C-t\sqrt{B^{2}-AC},$$ 
so that~\eqref{eq-isotropic-crpc} has form $F(f_{xx},f_{xy},f_{yy}) = 0$.
The function $F(A,B,C)$ is real analytic in the domain $D'\subset\mathbb{R}^3$ given by $B^{2}-AC>\varepsilon^2$ and $|F(A,B,C)|<\varepsilon$ (and complex analytic in its neighborhood $D\subset\mathbb C^{3}$).
We have 
\begin{equation}\label{eq-deteriso}
\frac{\partial F}{\partial A}\xi_1^2 + \frac{\partial F}{\partial B}\xi_1 \xi_2 + \frac{\partial F}{\partial C}\xi_2^2 =  \frac{\left(2 \sqrt{B^2-A C}+C t\right)\xi_1^2 -2 B t\xi_1 \xi_2 +\left(2 \sqrt{B^2-A C}+A t\right)\xi_2^2 }{2 \sqrt{B^2-A C}} \ne 0
\end{equation}
for all $(A,B,C)\in D'$ and $(\xi_1,\xi_2)\in\mathbb{R}^2\setminus\{(0,0)\}$, 
because the discriminant of the quadratic form 
is
$$-\frac{1}{4}t^2-\frac{A+C-t \sqrt{B^2-A C}}{2\sqrt{B^2-A C}}\cdot t -1 < 
-\frac{1}{4}t^2+\frac{\varepsilon}{2\varepsilon}|t|-1\le -\frac{3}{4}<0.
$$
Then~\eqref{eq-deteriso} holds in a neighborhood~$D$ as well. 
By Friedman's theorem (Theorem~\ref{th-petrowski-short}), $f$ is real analytic. 
\end{proof}

\begin{proof}[Proof of 
Theorem~\ref{thm-isotropic-equivalent}.] The existence of a 
family of $C^{2+1/2}$ functions $f(x,y,t)$ 
satisfying~\eqref{eq-isotropic-crpc} follows from Corollary~\ref{fmlessiso}. 
By compactness and Remark~\ref{rem-nonvanishing-det}, 
$f_{xx}f_{yy}-\left(f_{xy}\right)^2 < 0$ 
for sufficiently small $|t|$. Then by Friedman's theorem (Corollary~\ref{cor-crpr-analyticity-isotropic}), $f(x,y,t)$ is real analytic in $x$ and $y$ jointly. Thus, by Osgood’s lemma, $f(x,y,t)$ is real analytic in $x, y, t$ jointly. The uniqueness of the real analytic family $f(x,y,t)$ 
follows from Corollary~\ref{cor-uniqueness}. 
\end{proof}

Finally, Corollary~\ref{thm-plateau-iso} follows from Theorem~\ref{thm-isotropic-equivalent} and basic properties of harmonic functions.

\begin{proof}[Proof of Corollary~\ref{thm-plateau-iso}]
Let $\Gamma$ be the Jordan curve in the corollary. Since $\Gamma$ has no more than two common points (counting multiplicities) with any $z$-parallel plane, it follows
that $\Gamma$ has no $z$-parallel tangents, and the projection of $\Gamma$ to the $xy$ plane is the boundary of a convex Jordan domain $\Omega$ with analytic boundary. 

Since $\Gamma$ is analytic and has no $z$-parallel tangents, there is a continuous function $u\colon\overline\Omega\to \mathbb{R}$, harmonic in $\Omega$, whose graph spans $\Gamma$. Moreover, $u$ extends to a real analytic function $\tilde u\colon \tilde\Omega\to\mathbb{R}$, where $\tilde\Omega$ is a scaling of $\Omega$ centered at an interior point with a scaling factor $1+\varepsilon$ for a small $\varepsilon>0$. In particular, $u$ is a real analytic function in the closed domain~$\overline\Omega$.

Since $\Gamma$ has no more than four 
common points (counting multiplicities) with any plane,
the same is true for any analytic curve $\tilde\Gamma$ sufficiently close in the $C^\infty$ norm. 
Take $\tilde\Gamma$ to be the graph of the restriction of $\tilde u$ to $\partial\tilde\Omega$.
If the graph of $u$ had a flat point, then the tangent plane at this point would intersect $\tilde\Gamma$ at six or more points (this simple assertion is proved \cite[\S373]{nitsche-1975} in the case when $\tilde\Omega$ is a unit disc, but the same argument applies to any $\tilde\Omega$). Thus, $u$ has a nonvanishing Hessian (even on $\partial\Omega$).

The application of Theorem~\ref{thm-isotropic-equivalent} to the function $u$ concludes the proof.
\end{proof}

\section{Euclidean CRPC surfaces}\label{sec-euc}

The proof of Theorem~\ref{thm-euclidean} is parallel to the one of 
Theorem~\ref{thm-isotropic-equivalent}, 
only~\eqref{eq-isotropic-crpc} is replaced by the equation
\begin{equation}\label{eq-euclidean-crpc}
  (1+f_y^2)f_{xx}-2f_xf_yf_{xy}+(1+f_x^2)f_{yy}=t\sqrt{\left(1+f_x^2+f_y^2\right)\left(f_{xy}^2-f_{xx}f_{yy}\right)},
\end{equation}
equivalent to condition $4H^2 = -t^2K$ for the surface $z=f(x,y,t)$. We search for a solution in the form of series~\eqref{eq-expansion}. 
For any algebraic expression $F=F(f_{xx},f_{xy},f_{yy},f_{x},f_{y})$, such as $H$ and $K$, denote by $F^{(m)}$ its $m$-th derivative with respect to $t$ evaluated at $t = 0$. 
However, when the convergence of~\eqref{eq-expansion} is not assumed (e.g., in Corollary~\ref{cor-minimal-shauder} and Lemma~\ref{l-one-step-euc} below), then $f^{(m)}$ denotes a sequence of functions, not necessarily derivatives of one function $f(x,y,t)$,
and $F^{(0)}$ means $F\left(f^{(0)}_{xx},f^{(0)}_{xy},f^{(0)}_{yy},f^{(0)}_{x},f^{(0)}_{y}\right)$.

We fix a Jordan domain $\Omega$ with an analytic boundary, a parameter $\varepsilon>0$, and use the notation from Definition~\ref{def-holder}.
We start with a recurrence for the coefficients of series~\eqref{eq-expansion}. 

\begin{lemma}\label{l-euc} 
Assume that for a sequence of $C^2$ functions $f^{(m)}\colon\Omega \to\mathbb{R}$, we have $f^{(0)}_{xx}f^{(0)}_{yy}-\left(f^{(0)}_{xy}\right)^2< 0$, and series~\eqref{eq-expansion} $C^2$-converges for $|t|<\varepsilon$. Then the sum
satisfies~\eqref{eq-euclidean-crpc} if and only if $H^{(0)}=0$ and for each $m=1,2,\dots,$ we have
\begin{equation}\label{eq-elliptic}
\begin{split}
H_{0}(f^{(m)})  =
-\sum_{r=1}^{m-1}\binom{m}{r}H_{r}(f^{(m-r)}) + \frac{-mK^{(m-1)}-\frac{4}{m+1}\sum_{r=2}^{m-1}\binom{m+1}{r}
H^{(r)}H^{(m-r+1)}}{2^{1-\delta_{m1}}\sqrt{-K^{(0)}}},
\end{split}
\end{equation}
where the linear differential operator ${H}_{r}(u)$ is defined in terms of the left side of~\eqref{eq-expansion} as
\begin{multline}\label{hru}
 H_{r}(u) = \left(1+f_y^2\right)^{(r)}u_{xx}-2 \left(f_xf_y\right)^{(r)}u_{xy} +\left(1+f_x^2\right)^{(r)}u_{yy}  \\
+2^{\delta_{r0}}\left(f_{yy}^{(0)}f_x^{(r)} -f_{xy}^{(0)}f_y^{(r)}\right) u_x
 +  2^{\delta_{r0}}\left(f_{xx}^{(0)}f_y^{(r)} -f_{xy}^{(0)}f_x^{(r)}\right) u_y, 
\end{multline}
and the \emph{normalized mean and Gaussian curvatures} are defined as 
\begin{equation}\label{H_and_K_euc}
    \begin{split}
       H = \frac{1}{2}\left((1+f_y^2)f_{xx}-2f_xf_yf_{xy}+(1+f_x^2)f_{yy}\right) \quad\text{ and }\quad K = - \left(1+f_x^2+f_y^2\right)\left(f_{xy}^2-f_{xx}f_{yy}\right). 
    \end{split}
\end{equation}
\end{lemma}

Beware that throughout this section, by slight abuse of notation, $H$ and $K$ denote \emph{normalized} mean and Gaussian curvatures (by a suitable power of $1+f_x^2 + f_y^2$) rather than the curvatures themselves. 
\begin{proof}
Since series~\eqref{eq-expansion} $C^2$-converges, 
it can be termwise differentiated with respect to $x$ and $y$ two times.
Hence $H$ and $K$ are real analytic in $t$ 
for each $(x,y)\in\Omega$.
Moreover, \eqref{eq-euclidean-crpc} is
equivalent to $2H=t\sqrt{-K}$. 
By Lemma~\ref{der-mean}, ~\eqref{eq-euclidean-crpc} is equivalent to $H^{(0)}=0$  and~\eqref{H_m}, 
with the `$+$' sign in $\pm$
because $2H^{(1)}=\sqrt{-K^{(0)}}$. 
It remains to prove that 
\begin{equation}\label{Hm1}
   2 H^{(m)} = H_{0}\left(f^{(m)}\right) + \sum_{r=1}^{m-1}\binom{m}{r}H_{r}\left(f^{(m-r)}\right) \qquad \text{for each } m=1,2,\dots.
\end{equation}

Indeed, by the Leibniz rule, we get
\begin{align*}
2H^{(m)} &= \sum_{r=0}^{m}\binom{m}{r}\left[\left(1+f_y^2\right)^{(r)}f_{xx}^{(m-r)}-2\left(f_xf_y\right)^{(r)}f_{xx}^{(m-r)}+\left(1+f_x^2\right)^{(r)}f_{yy}^{(m-r)}\right] \\
&=\left(1+f_y^2\right)^{(0)}f_{xx}^{(m)}-2\left(f_xf_y\right)^{(0)}f_{xy}^{(m)}
+\left(1+f_x^2\right)^{(0)}f_{yy}^{(m)}  + \\
&\quad\sum_{r=1}^{m-1}\binom{m}{r}\left[\left(1+f_y^2\right)^{(r)}f_{xx}^{(m-r)}-2\left(f_xf_y\right)^{(r)}f_{xy}^{(m-r)}+\left(1+f_x^2\right)^{(r)}f_{yy}^{(m-r)}\right] +\\
&\quad\left(1+f_y^2\right)^{(m)}f_{xx}^{(0)}-2\left(f_xf_y\right)^{(m)}f_{xy}^{(0)}+\left(1+f_x^2\right)^{(m)}f_{yy}^{(0)}\\
&=\left(1+f_y^2\right)^{(0)}f_{xx}^{(m)}-2 \left(f_xf_y\right)^{(0)}f_{xy}^{(m)}+\left(1+f_x^2\right)^{(0)}f_{yy}^{(m)}\\
&\quad +\sum_{r=1}^{m-1}\binom{m}{r}\left[\left(1+f_y^2\right)^{(r)}f_{xx}^{(m-r)}-
2\left(f_xf_y\right)^{(r)}f_{xy}^{(m-r)}+\left(1+f_x^2\right)^{(r)}f_{yy}^{(m-r)}\right.\\
&\quad +\left.\left(f_{yy}^{(0)}f_x^{(r)} -f_{xy}^{(0)}f_y^{(r)}\right) f_{x}^{(m-r)}
+\left(f_{xx}^{(0)}f_y^{(r)} -f_{xy}^{(0)}f_x^{(r)}\right) f_{y}^{(m-r)}\right] \\
&\quad +2 \left[f_{y}^{(m)}f_{y}^{(0)}f_{xx}^{(0)}-f_{x}^{(m)}f_{y}^{(0)}f_{xy}^{(0)} - f_{y}^{(m)}f_{x}^{(0)}f_{xy}^{(0)} + f_{x}^{(m)}f_{x}^{(0)}f_{yy}^{(0)}\right] \\
&= H_{0}\left(f^{(m)}\right) + \sum_{r=1}^{m-1}\binom{m}{r}H_{r}\left(f^{(m-r)}\right).\\[-1.1cm]
\end{align*}
\end{proof}

We need the following particular case of Schauder's existence theorem and 
estimates. 

\begin{theorem}[Schauder's existence theorem and estimates] \label{thm-schauder-existance} \textup{(See \cite[Theorem~7.3 and Remark~2 after it]{Agmon-etal-59} and~\cite[Theorems~6.6, 6.14, and end of \S6.3]
{Gilbarg-Trudinger-83})}
Let $C^{1/2}$ functions $a_1, a_2, a_{11}, a_{12}, a_{22}\colon\overline{\Omega} \rightarrow  \mathbb{R}$ satisfy 
\begin{equation}\label{stricteliptic}
    a_{11}(\mathbf{x})\xi_1^2 + a_{12}(\mathbf{x})\xi_{1} \xi_2 + a_{22}(\mathbf{x})\xi_2^2\ge \lambda \cdot (\xi_1^2+\xi_2^2)
\end{equation}
for some $\lambda>0$ and all 
$\xi_1, \xi_2 \in \mathbb{R}, \mathbf{x} \in \overline{\Omega}$.
Then for each $C^{1/2}$ function $g\colon\overline{\Omega} \rightarrow \mathbb{R}$, there exists a unique $C^{2 + 1/2}$ function $f\colon\overline{\Omega} \rightarrow \mathbb{R}$ that vanishes on $\partial \Omega$ and satisfies 
$$ a_1(\mathbf{x})f_x(\mathbf{x}) + a_2(\mathbf{x})f_y(\mathbf{x}) + a_{11}(\mathbf{x}) f_{xx}(\mathbf{x}) + a_{12}(\mathbf{x})f_{xy}(\mathbf{x})  +a_{22}(\mathbf{x}) f_{yy}(\mathbf{x}) = g(\mathbf{x}) \qquad \text{ for each } \mathbf{x}\in {\Omega}.$$
Moreover, there is a constant 
$C(\Omega,a_1, a_2, a_{11}, a_{12}, a_{22})>0$
depending on $\Omega,a_1, a_2, a_{11}, a_{12}, a_{22}$ 
such that 
$$\|f\|_{C^{2 +1/2}} \leq C(\Omega,a_1, a_2, a_{11}, a_{12}, a_{22})\cdot \|g\|_{C^{1/2}}.$$
\end{theorem}

\begin{corollary}\label{cor-minimal-shauder}
For any $C^{2+1/2}$ function $f^{(0)}\colon\overline\Omega\to\mathbb{R}$ and $C^{1/2}$ function $g:\overline{\Omega} \rightarrow \mathbb{R}$, there exists a unique $C^{2 + 1/2}$ function $u\colon\overline{\Omega} \rightarrow \mathbb{R}$ that vanishes on $\partial \Omega$ and satisfies $H_0(u)=g$, under notation~\eqref{hru}. Moreover, there exists a constant 
$C\left(\Omega, f^{(0)}\right)>0$ depending on $\Omega$ and $f^{(0)}$ such that 
$$\|u\|_{C^{2 +1/2}} \leq C\left(\Omega,f^{(0)}\right)\|g\|_{C^{1/2}}.$$
\end{corollary}

\begin{proof} 
    The coefficients $a_i$ and $a_{ij}$ of the operator $H_0(u)$ are $C^{1/2}$ functions as products of $C^{1/2}$ functions.
    The operator $H_0(u)$ 
    satisfies~\eqref{stricteliptic} 
    because
$$
\left(1+f_y^2\right)^{(0)}\xi_1^2-2 \left(f_xf_y\right)^{(0)}\xi_1\xi_2 +\left(1+f_x^2\right)^{(0)}\xi_2^2
=
\xi_1^2 +\xi_2^2 +   \left(f^{(0)}_y\xi_1 - f^{(0)}_x\xi_2\right)^2
\ge
\xi_1^2 +\xi_2^2. 
$$ 
%
By Theorem~\ref{thm-schauder-existance}, the corollary follows.
\end{proof}

\begin{corollary}\label{cor-uniqueness-euc}
Under the assumptions of Theorem~\ref{thm-euclidean}, there is no more than one (up to restriction) real analytic family of surfaces $\Phi^s$ with the ratio $s-1$ of principal curvatures and $\partial \Phi^s=\partial \Phi^0$.
\end{corollary}

\begin{proof} 
Since $\Phi^0$ is the graph of a real analytic function, it follows that $\Phi^s$ is, for sufficiently small $|s|$. Indeed, consider the composition of an injective regular real analytic parametrization $D^2\to \Phi^s$ and the projection 
to $xy$-plane. For $s=0$, it is a $C^1$ embedding. Since a $C^1$ approximation of an embedding is again an embedding (see \cite[Theorem~3.10]{munkres-1966}), it follows that the composition is injective and has a nondegenerate differential for all sufficiently small $|s|$. Since $\partial\Phi^s=\partial\Phi^0$, it follows that the composition has the same image for all such~$s$.

By the inverse function theorem, the family $\Phi^s$ is graphs of a real analytic family of functions $f\colon\overline\Omega\times [-\varepsilon;\varepsilon]\to\mathbb{R}$ for some $\Omega$ and $\varepsilon>0$. The functions 
satisfy~\eqref{eq-euclidean-crpc}, where $t=s/\sqrt{1-s}$,
and $f(x,y,t)=f^{(0)}(x,y)$ for $t=0$ or $(x,y)\in\partial\Omega$. Hereafter, $f^{(m)}$ is the $m$-th derivative of $f(x,y,t)$ with respect to $t$ evaluated at $t=0$. 

It remains to prove that $f^{(m)}$ is uniquely determined by $f^{(0)}$ for 
$m\ge 1$, so is 
$f(x,y,t)=\sum_{m=1}^\infty\frac{t^m}{m!}\,f^{(m)}(x,y)$. We prove it by induction on $m$. 
Assume that 
$f^{(0)},\dots,f^{(m-1)}$ have already been determined. Since $f(x,y,t)=f^{(0)}(x,y)$ on 
$\partial\Omega$, it follows that $f^{(m)}(x,y)=0$ on $\partial\Omega$. 
Since the graph of $f^{(0)}$ is a minimal surface without flat points, by Remark~\ref{rem-nonvanishing-det} it follows that $f^{(0)}_{xx}f^{(0)}_{yy}-(f^{(0)}_{xy})^2<0$.
By compactness, series~\eqref{eq-expansion} $C^2$-converges for sufficiently small $t$, because the Taylor series of the derivatives of $f(x,y,t)$ converge uniformly in a neighborhood of each point of $\overline\Omega\times\{0\}$.
Thus, by Lemma~\ref{l-euc}, we get~\eqref{eq-elliptic}.
This is a linear PDE in $f^{(m)}$ with the right side depending only on $f^{(0)},\dots,f^{(m-1)}$. By 
Corollary~\ref{cor-minimal-shauder}, this equation with the zero Dirichlet boundary condition has no more than one solution. Thus, $f^{(m)}$ is uniquely determined. By induction, the corollary follows.
\end{proof}


For the existence of the family $f(x,y,t)$, 
we need
a recursive bound for $\|f^{(m)}\|:=\|f^{(m)}\|_{C^{2+1/2}}$.

\begin{lemma} \label{l-one-step-euc} 
Starting with a $C^{2+1/2}$ function $f^{(0)}\colon\overline\Omega\to\mathbb{R}$,
which has 
a non-vanishing Hessian and 
$H^{(0)}=0$, 
define a sequence of $C^{2+1/2}$ functions $f^{(m)}\colon\overline\Omega\to\mathbb{R}$, where $m\ge 1$, recursively as unique solutions of~\eqref{eq-elliptic} vanishing on $\partial\Omega$. 
Then 
there is a constant 
$C\left(\Omega,f^{(0)}\right)>0$ such that 
\begin{multline*}
\left\|f^{(m)}\right\| \le  C\left(\Omega,f^{(0)}\right)\cdot 
 \left(\sum_{0\le p_1,p_2,p_3 <m \atop \sum p_i = m}^{}\binom{m}{p_1,p_2, p_3 }\prod_{i=1}^{3}\left\|f^{(p_i)}\right\| + \right.\\ \left. m\sum_{q_1, \cdots, q_4 \ge 0 \atop \sum q_i = m-1}^{}\binom{m-1}{q_1,\dots, q_4}\prod_{i=1}^{4}\left\|f^{(q_i)}\right\| + \frac{1}{m+1}\sum_{0\le r_1,\dots,r_6 <m\atop \sum r_i = m+1 }^{}\binom{m+1}{r_1,\dots, r_6 }\prod_{i=1}^{6}\left\|f^{(r_i)}\right\|\right)
 \qquad\text{for $m \ge 1$.}
\end{multline*}
\end{lemma}

\begin{proof}
Denote $\left\|f\right\|{}':=\left\|f\right\|_{C^{1/2}}$. 
By Remark~\ref{rem-nonvanishing-det}, $K^{(0)}<-\varepsilon^2$ in $\overline{\Omega}$ for some $\varepsilon>0$ so that $\left\|{1}/{\sqrt{-K^{(0)}}}\right\|'<\infty$.
By Schauder's estimates (Corollary~\ref{cor-minimal-shauder}) and the submultiplicativity of the H\"older norm, we get, for some $C_1$ and $C_2$,
\begin{align*}
\textstyle \left\|f^{(m)}\right\| 
&\le  C_1\left(\Omega,f^{(0)}\right)\left\|-\sum_{r=1}^{m-1}\binom{m}{r}H_{r}(f^{(m-r)}) + \frac{-mK^{(m-1)}-\frac{4}{m+1}\sum_{r=2}^{m-1}\binom{m+1}{r}
H^{(r)}H^{(m-r+1)}}{2^{1-\delta_{m1}}\sqrt{-K^{(0)}}}\right\|' \\
&\le C_{2}\left(\Omega, f^{(0)}\right)
\left(\sum_{r=1}^{m-1}\binom{m}{r}\left\|H_{r}(f^{(m-r)})\right\|'\right. +m\left\|K^{(m-1)}\right\|'  
\left.\textstyle+\frac{1}{m+1}\sum_{r=2}^{m-1}\binom{m+1}{r}
\left\| H^{(r)}H^{(m-r+1)}\right\|'\right).
\end{align*}

Let us bound the first summand on the right side. 
For $0 < r < m $, by 
the Leibniz rule, we get, for some $C_3$,
\begin{align*}
    \left\|H_{r}(f^{(m-r)})\right\|' & \le C_3(\Omega)\left(\left\|\left(f_y^2\right)^{(r)}\right\|'\left\|f_{xx}^{(m-r)}\right\|'+2\left\|\left(f_xf_y\right)^{(r)}\right\|'\left\|f_{xy}^{(m-r)}\right\|'+\left\|\left(f_x^2\right)^{(r)}\right\|'\left\|f_{yy}^{(m-r)} \right\|'+\right. \\
     &\quad \textstyle \left.\left(  \left\| f_{yy}^{(0)}\right\|'\left\| f_x^{(r)}\right\|'+ \left\| f_{xy}^{(0)}\right\|'\left\| f_y^{(r)} \right\|'\right) \left\|f_{x}^{(m-r)}\right\|'+\left(  \left\| f_{xx}^{(0)}\right\|'\left\| f_y^{(r)}\right\|'+ \left\| f_{xy}^{(0)}\right\|'\left\| f_x^{(r)} \right\|'\right) \left\|f_{y}^{(m-r)}\right\|'\right) \\
     & \le 4C_3(\Omega)\left\|f^{(m-r)}\right\| \cdot \sum_{i=0}^{r}\binom{r}{i}\left\|f^{(i)} \right\| \left\|f^{(r-i)} \right\| + 4C_3(\Omega)\left\|f^{(0)}\right\| \left\|f^{(r)}\right\| \left\|f^{(m-r)}\right\|\\
     & = 4C_3(\Omega)\sum_{0\le p_1,p_2,p_3 <m \atop  p_1 = m-r, \sum p_i = m}^{}\binom{r}{p_2}\prod_{i=1}^{3}\left\|f^{(p_i)}\right\| + 4C_3(\Omega)\sum_{0\le p_1,p_2,p_3 <m \atop  p_1 = m-r, p_2 =r, \sum p_i = m}^{}\binom{r}{p_2 }\prod_{i=1}^{3}\left\|f^{(p_i)}\right\|.
\end{align*}
Thus
\begin{equation*}
\sum_{r=1}^{m-1}\binom{m}{r}\left\|H_{r}(f^{(m-r)})\right\|'\le 8C_3(\Omega)  \sum_{0\le p_1,p_2,p_3 <m \atop \sum p_i = m}^{}\binom{m}{p_1,p_2, p_3 }\prod_{i=1}^{3}\left\|f^{(p_i)}\right\|.
\end{equation*}
Analogously, we bound the second and third summands 
for suitable $C_4\left(\Omega,\frac{1}{\left\|f^{(0)}\right\|}\right)>0$ and $C_5\left(\Omega,\frac{1}{\left\|f^{(0)}\right\|}\right)>0$:
\begin{align*}
    \left\|K^{(m-1)}\right\|'& \le  C_4\left(\Omega,\frac{1}{\left\|f^{(0)}\right\|}\right) \sum_{q_1, \cdots, q_4 \ge 0 \atop \sum q_i = m-1}^{}\binom{m-1}{q_1,\dots, q_4}\prod_{i=1}^{4}\left\|f^{(q_i)}\right\|, \\
    \sum_{r=2}^{m-1}\binom{m+1}{r}
\left\|H^{(r)}H^{(m-r+1)}\right\|'& \le C_5\left(\Omega,\frac{1}{\left\|f^{(0)}\right\|}\right)\sum_{0\le r_1,\dots,r_6 <m\atop \sum r_i = m+1 }^{}\binom{m+1}{r_1,\dots, r_6 }\prod_{i=1}^{6}\left\|f^{(r_i)}\right\|.
\end{align*}
Here, the lower-degree terms are absorbed at the price of multiplying by a power of $\left\|f^{(0)}\right\|$; thus, $C_4$ and $C_5$ depend on $\|f^{(0)}\|$.
Combining all the resulting bounds, we complete the proof.
\end{proof}

Now we bound $\left\|f^{(m)}\right\|$ by the Taylor coefficients of a specially constructed analytic function. Abbreviate $M:=\left\|f^{(0)}\right\|,
N:=C\left(\Omega, f^{(0)}\right)$.

\begin{lemma} \label{l-reccurence-euc} Let $M,N>0$. 
Define a sequence $a^{(m)}=a^{(m)}(M,N)$ recursively by $a^{(0)}:=M$ and
\begin{multline*}
a^{(m)} := N\sum_{ 0\le p_1,p_2,p_3 <m \atop \sum p_i = m   }^{}\binom{m}{p_1,p_2, p_3 }\prod_{i=1}^{3}a^{(p_i)}+\\+
Nm\sum_{q_1,.., q_4 \ge 0 \atop \sum q_i = m-1}^{}\binom{m-1}{q_1,.., q_4}\prod_{i=1}^{4}a^{(q_i)}+ \frac{N}{m+1}\sum_{0\le r_1,\dots,r_6 <m\atop \sum r_i = m+1 }^{}\binom{m+1}{r_1,.., r_6 }\prod_{i=1}^{6}a^{(r_i)} \qquad\text{for $m\ge 1$.}
\end{multline*}
Then $\sum_{m=0}^{\infty}a^{(m)}t^m/m!$ converges
in a neighborhood of $0$, and the sum is the inverse function~of 
\begin{equation}\label{eq-t(a)}
t(a) :=\frac{(a-M)\left(p(a)+ \sqrt{p(a)^2-4 N^2 \left(a^4+15 M^{12} N^2\right)\left(a^4+2 a^3 M+3 a^2 M^2+4 a M^3+5 M^4\right) }\right)}{2 \left(a^4 N+15 M^{12} N^3\right)} \,
\end{equation}
where $p(a):=1 - (a - M) (a + 2 M) N + 30 M^8 N^2$, in a neighborhood of $a=M$.
\end{lemma}

\begin{proof}
 Denote $a^{(m)} =m!\cdot b_{m}$ for each $m\ge 0$,   and introduce formal power series $a(t) =  \sum_{m=0}^{\infty}a^{(m)}t^m/m!$. Then 
 \begin{align*}
    b_{m} &= N \sum_{ 0\le p_1,p_2,p_3 <m \atop \sum p_i = m   }^{}\prod_{i=1}^{3}b_{p_i}+ N\sum_{q_1,.., q_4 \ge 0 \atop \sum q_i = m-1}^{}\prod_{i=1}^{4}b_{q_i}+ N\sum_{0\le r_1,\dots,r_6 <m\atop \sum r_i = m+1 }^{}\prod_{i=1}^{6}b_{r_i} \qquad \text{ for } m\ge 1, \\
    a(t) &= b_0 + N\sum_{m=1}^{\infty}  \sum_{ 0\le p_1,p_2,p_3 <m \atop \sum p_i = m   }^{}\left(\prod_{i=1}^{3}b_{p_i}\right)t^m+ N\sum_{m=1}^{\infty}\sum_{q_1,.., q_4 \ge 0 \atop \sum n_i = m-1}^{}\left(\prod_{i=1}^{4}b_{q_i}\right)t^m+ N\sum_{m=1}^{\infty}\sum_{0\le r_1,\dots,r_6 <m\atop \sum r_i = m+1 }^{}\left(\prod_{i=1}^{6}b_{r_i}\right)t^m. 
 \end{align*}
 
 Let us rewrite the second summand of $a(t)$, using $b_0 = a^{(0)} = M$ and the properties of formal power series:
 \begin{align*}
     \sum_{m=1}^{\infty}\sum_{ 0\le p_1,p_2,p_3 <m \atop \sum p_i = m   }^{}&\left(\prod_{i=1}^{3}b_{p_i}\right)t^m   = \sum_{m=1}^{\infty}\left(\sum_{ 0\le p_1,p_2,p_3 \atop \sum p_i = m   }^{}\prod_{i=1}^{3}b_{p_i}-3b_{0}^2b_{m}\right)t^m \\
     = &\sum_{m=0}^{\infty}\left(\sum_{ 0\le p_1,p_2,p_3 \atop \sum p_i = m   }^{}\prod_{i=1}^{3}b_{p_i}-3b_{0}^2b_{m}\right)t^m + 2b_{0}^3
     = a(t)^3 -3b_{0}^2a(t) + 2b_{0}^3 = \left(a(t)-M\right)^2\left(a(t) + 2M\right).
\end{align*}
Analogously, we rewrite the third and fourth summands of  $a(t)$:
\begin{equation*}
    \sum_{m=1}^{\infty} \sum_{q_1,.., q_4 \ge 0 \atop \sum q_i = m-1}^{}\left(\prod_{i=1}^{4}b_{q_i}\right)t^m = ta(t)^4
    \end{equation*}
    and
    $$\sum_{m=1}^{\infty} \sum_{0\le r_1,\dots,r_6 <m\atop \sum r_i = m+1 }^{}\left(\prod_{i=1}^{6}b_{r_i}\right)t^m = \frac{1}{t} \left(a(t)^6-6 a(t) \left(5 M^8 N t+M^5\right)+5 \left(3 M^{12} N^2 t^2+6 M^9 N t+M^6\right)\right).$$
We get the following equation for $a = a(t)$:
\begin{equation}\label{eq:tofa}
\left(a^4 N+15 M^{12} N^3\right)t^2 -t (a-M) \left(1 - (a - M) (a + 2 M) N + 30 M^8 N^2\right) + a^6 N-6 a M^5 N+5 M^6 N = 0.
\end{equation}
Solving this quadratic equation with respect to $t$, we get~\eqref{eq-t(a)} 
up to the sign in front of the square root. The polynomial inside the square root tends to $\bigl(30 M^{8}N^{2}+1\bigr)^{2}-60 M^{4}N^{2}\bigl(15 M^{12}N^{2}+M^{4}\bigr)=1$ as $a\to M$; in particular, it is positive for $a$ close to $M$. Thus $t'(M)=\frac{p(M)\pm 1}{2 \left(M^4 N+15 M^{12} N^3\right)}$. Since $1/t'(M)=a'(0)=a^{(1)}=M^4N$, we indeed have a plus sign in front of the square root in~\eqref{eq-t(a)}.
By the analyticity of $t(a)$ in a neighborhood of $a(0)=M$ and the condition $t'(a(0)) =1/(M^4N)\neq 0$, its inverse function $a(t)$ is analytic and its Taylor series converges in a neighborhood of $0$.
\end{proof}

By induction on $m$, Lemmas~\ref{convergence},~\ref{l-euc},~\ref{l-one-step-euc}, and~\ref{l-reccurence-euc} imply the following. 

\begin{corollary}\label{fmlesseuc} Under the assumptions and notation of Lemmas~\ref{l-one-step-euc} and~\ref{l-reccurence-euc}, for each $m\ge 0$ we have
\begin{equation*}
  \|f^{(m)}\|\le a^{(m)}\left(
  \left\|f^{(0)}\right\|,
  C\left(\Omega, f^{(0)}\right)\right).
\end{equation*}
Hence, series~\eqref{eq-expansion} converges in the $C^{2+1/2}$ norm for 
sufficiently small $|t|$, and the sum satisfies~\eqref{eq-euclidean-crpc}.
\end{corollary}


\begin{remark}
This corollary and its proof remain true if we only assume that $\partial\Omega$ and $f^{(0)}\colon\overline\Omega\to\mathbb{R}$
are of class $C^{2+1/2}$ (instead of the analyticity), so that we get a $C^{2+1/2}$ solution to~\eqref{eq-euclidean-crpc}. 
\end{remark}

The analyticity of the resulting solution to~\eqref{eq-euclidean-crpc} is deduced from Friedman's theorem. Recall that CPRC surfaces are not allowed to have flat points according to the definition in Section~\ref{sec-contributions}.

\begin{corollary}[Analyticity of CPRC surfaces] \label{cor-crpr-analyticity}
    Any $C^{2+1/2}$ function $f\colon \overline\Omega\to \mathbb{R}$ satisfying equation~\eqref{eq-euclidean-crpc} for some $t$ and the inequality $f_{xx}f_{yy}-f^2_{xy}<0$ is real analytic in the closed domain. So, any $C^{2+1/2}$ CPRC surface with $K<0$ bounded by an analytic Jordan curve is an analytic surface with analytic boundary.
\end{corollary}

\begin{proof} By compactness, there exists $\varepsilon>0$ such that 
$f_{xx}f_{yy}-\left(f_{xy}\right)^2 < -\varepsilon^2$. Introduce the function 
$$F(A,B,C, D, E)= \left(1 + E^2\right)A-2BDE + \left(1 + D^2\right)C-t \sqrt{\left(1 + D^2+E^2\right) \left(B^2-AC\right)},$$ 
so that~\eqref{eq-euclidean-crpc} has form $F(f_{xx},f_{xy},f_{yy}, f_{x}, f_{y}) = 0$.
The function $F(A,B,C, D, E)$ is real analytic in the domain $D'\subset\mathbb{R}^5$ given by $B^{2}-AC>\varepsilon^2$ and $|F|<\varepsilon$ (and complex analytic in its neighborhood $D''\subset\mathbb C^{5}$).

For all $(A,B,C, D, E)\in D'$ and $(\xi_1,\xi_2)\in\mathbb{R}^2\setminus\{(0,0)\}$, we have 
\begin{equation}\label{eq-detereuc}
\textstyle \frac{\partial F}{\partial A}\xi_1^2 + \frac{\partial F}{\partial B}\xi_1 \xi_2 + \frac{\partial F}{\partial C}\xi_2^2 =
\frac{2\sqrt{B^2-A C}\left((1+E^2)\xi_1^2 -2DE\xi_1 \xi_2 +(1+D^2)\xi_2^2  \right) + t \sqrt{D^2+E^2+1}\left(C \xi_1^2 -2B\xi_1 \xi_2 +A\xi_2^2\right) }{2 \sqrt{B^2-A C}} \ne 0,  
\end{equation}
because the discriminant of the quadratic form 
is
\begin{equation*}
-\left(\frac{t^2}{4}+1\right) \left(D^2+E^2+1\right) -\frac{F(A,B,C,D,E)\sqrt{D^2+E^2+1}}{2\sqrt{B^2-A C}}\cdot t
<\left(-\frac{t^2}{4}-1+ \frac{\varepsilon}{2\varepsilon}|t| \right) \left(D^2+E^2+1\right) \le -\frac{3}{4}<0.
\end{equation*}
Then~\eqref{eq-detereuc} holds in a neighborhood~$D''$ as well. By Friedman's theorem (Theorem~\ref{th-petrowski-short}), $f$ is real analytic. 
\end{proof}

\begin{proof}[Proof of Theorem~\ref{thm-euclidean}.] 
Let the minimal surface be the graph of a real analytic function $f^{(0)}\colon \overline\Omega\to\mathbb{R}$.
By Corollary~\ref{fmlesseuc}, there is a family of $C^{2+1/2}$ functions $f(x,y,t)$ satisfying~\eqref{eq-euclidean-crpc} such that $f(x,y,t)=f^{(0)}(x,y)$ for $t=0$ or $(x,y)\in\partial\Omega$.
By Remark~\ref{rem-nonvanishing-det} and a compactness argument, 
$f_{xx}f_{yy}-\left(f_{xy}\right)^2 < 0$ 
for sufficiently small $|t|$. Then by Friedman's theorem (Corollary~\ref{cor-crpr-analyticity}), $f(x,y,t)$ is real analytic in $x$ and $y$ jointly. Thus, by Osgood’s lemma, $f(x,y,t)$ is real analytic in $x, y, t$ jointly. Then $z=f(x,y,s/\sqrt{1-s})$ is the desired family of CRPC surfaces. The uniqueness of the analytic family of CRPC surfaces is proved in Corollary~\ref{cor-uniqueness-euc}.
\end{proof}

Now we deduce Corollary~\ref{thm-plateau-euc} from Theorem~\ref{thm-euclidean} and a few known results on Plateau's problem.

\begin{proof}[Proof of Corollary~\ref{thm-plateau-euc}]
Let $\Gamma$ be the Jordan curve in the corollary. Since $\Gamma$ has no more than two common points (counting multiplicities) with any $z$-parallel plane, it follows
that $\Gamma$ has no $z$-parallel tangents, and the projection of $\Gamma$ to the $xy$ plane is the boundary of a convex Jordan domain $\Omega$ with analytic boundary. 

Then by Finn's criterion \cite[Theorem 13.14]{Gilbarg-Trudinger-83}, there is a $C^2$ function $u\colon\overline\Omega\to \mathbb{R}$, whose graph $\Phi$ is a minimal surface spanning $\Gamma$. 
Since $\Gamma$ is analytic, by Lewy's theorem \cite[\S334]{nitsche-1975}, 
the minimal surface extends beyond $\Gamma$. The tangent planes 
at $\Gamma$ are not $z$-parallel because $u$ is a $C^2$ function in the closed domain $\overline\Omega$. Thus, the extended surface contains the graph $\tilde\Phi$ of a real analytic function $\tilde u\colon \tilde\Omega\to\mathbb{R}$, where $\tilde\Omega$ is a scaling of $\Omega$ centered at an interior point with a scaling factor $1+\varepsilon$ for a small $\varepsilon>0$. In particular, $u$ is a real analytic function in the closed domain~$\overline\Omega$.

Since $\Gamma$ has no more than four 
common points (counting multiplicities) with any plane,
the same is true for any analytic curve $\tilde\Gamma$ sufficiently close in $C^\infty$ norm. Applying this for $\tilde\Gamma=\partial\tilde\Phi$ and using \cite[\S404]{nitsche-1975}, we conclude that $\tilde\Phi$ has no interior flat points. Thus, $\Phi$ has no flat points (even on $\Gamma$).

Application of Theorem~\ref{thm-euclidean} to the surface $\Phi$ concludes the proof.
\end{proof}

It would be interesting to investigate the universality of this construction:

\begin{problem} Is each CRPC surface contained in one family of CRPC surfaces with a minimal surface? 
\end{problem}

It is interesting when Plateau's problem for 
CRPC surfaces can have uncountably many solutions:

\begin{problem} (Cf.~Example~\ref{ex-catenoid}) Is there an analytic curve that can be spanned by uncountably many surfaces with the given ratio $a$ of principal curvatures, for each $a\ne -1$ close to $-1$? 
\end{problem}

\section{Approximate CRPC surfaces}\label{sec-proofs-variations}



In this section, we 
derive a closed-form expression for \emph{second-order approximation} of an isotropic CRPC surface, that is, for the first three terms of~\eqref{eq-expansion} (Proposition~\ref{p-approximation}). This expression 
makes sense even in the presence of a flat point, but of even ``multiplicity'' only. Note that our second-order approximation preserves the boundary of the surface only to zeroth order.
See Fig.~\ref{fig:3}.

\begin{figure}[htb]
    \centering
\begin{overpic}
[width=0.5\linewidth]{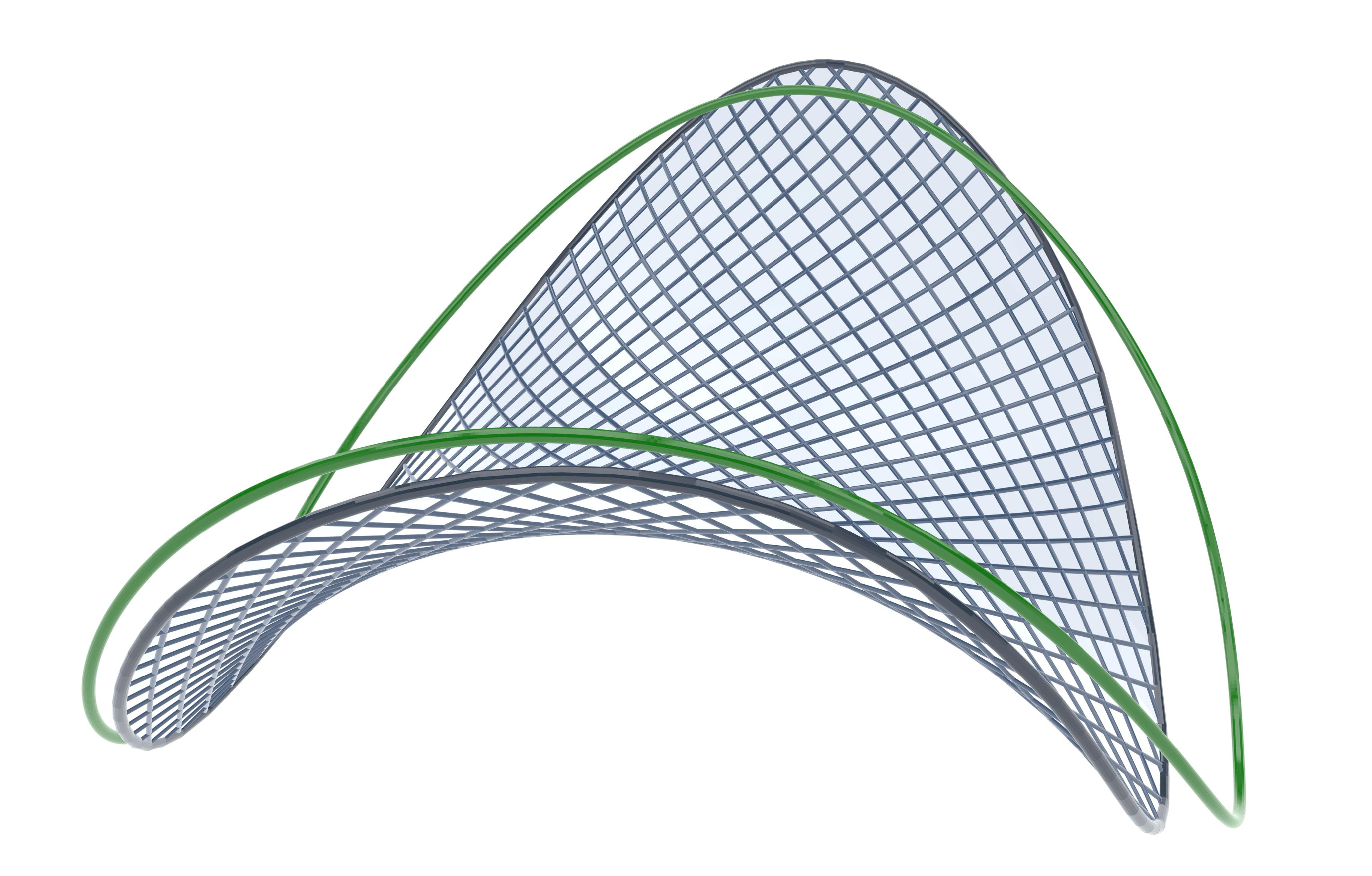}
\small
 \end{overpic}
\caption{An approximate solution to Plateau's problem for CPRC surfaces.
The asymptotic lines of the surface (gray) intersect at a constant angle ($80^\circ$). Fitting the prescribed boundary curve (green) is not perfect because our optimization algorithm gives higher preference to achieving the prescribed intersection angle (and positions of flat points if present).  
}
\label{fig:3}
\end{figure}

This expression is simple and elegant when using the complex coordinate $w=x+iy$ in the $xy$ plane. Using the Wirtinger derivatives \(f_{w}:=\tfrac12(f_{x}-if_{y})\), \(f_{\bar w}:=\tfrac12(f_{x}+if_{y})\), the isotropic mean and Gaussian curvatures of a surface $z=f(x,y)$ can be written as 
$
H=\frac{f_{xx}+f_{yy}}{2}=2\,f_{w\bar w}$ and $
K=f_{xx}f_{yy}-f_{xy}^{2}=4\,(f_{w\bar w}^{2}-f_{ww}f_{\bar w\bar w}).
$
Then the equation $2H=t\sqrt{-K}$ of isotropic CRPC surfaces takes an elegant form
\begin{equation}\label{eq-complex-crpc}
  f_{w\bar w}=\tilde t\,\sqrt{\,f_{ww}f_{\bar w\bar w}\,}, 
\end{equation}
where $\tilde t=t/\sqrt{4+t^2}$. In what follows, we omit the tilde and write $t$ instead of $\tilde t$ for simplicity. We construct an approximate solution of equation~\eqref{eq-complex-crpc} 
as follows. 

\begin{proposition}\label{p-approximation}
    Let $f^{(0)}$ be a harmonic function in a Jordan domain with a non-vanishing Hessian. Let $g(w)$ be a complex analytic function with real part $f^{(0)}(w)/2$, 
    and $h(w)$ be an antiderivative of a continuous branch of~$\sqrt{g''(w)}$.
    Then the function
    \begin{equation}\label{eq:approx}
        f(w,t) = 2\mathrm{Re}\, g(w) + |h(w)|^2\cdot t + \mathrm{Re}\,\left(h(w)^2\right)\log |h'(w)|\cdot t^2
    \end{equation}
    satisfies 
\begin{equation*}
     f_{w\bar w}=t\,\sqrt{\,f_{ww}f_{\bar w\bar w}\,}+ o\left(t^2\right) 
     \qquad \text{as } t\to 0.
\end{equation*}
\end{proposition}


The proof relies on the following analogue of Lemma~\ref{l-poisson} for equation~\eqref{eq-complex-crpc} proved similarly.

\begin{lemma} \label{l-complex-poisson} 
Assume that $f^{(0)}(w),\dots,f^{(n)}(w)$ are $C^2$ functions in a Jordan domain such that $f^{(0)}_{ww}f^{(0)}_{\bar w\bar w}>0$. Then 
$f(w,t):=\sum_{m=0}^n f^{(m)}(w)t^m/m!$
satisfies~$f_{w\bar w}=t\,\sqrt{\,f_{ww}f_{\bar w\bar w}\,}+ o\left(t^n\right)$ as $t\to 0$ if and only if 
\begin{equation}\label{eq-complex-poisson}
f^{(m)}_{w\bar w}=\frac{
m\sum_{r=0}^{m-1}\binom{m-1}{r}
f^{(r)}_{ww}f^{(m-r-1)}_{\bar w\bar w}
-\frac{1}{m+1}\sum_{r=2}^{m-1}\binom{m+1}{r}
f^{(r)}_{w\bar w}f^{(m-r+1)}_{w\bar w}
}{2^{1-\delta_{m1}}\sqrt{f^{(0)}_{ww}f^{(0)}_{\bar w\bar w}}}
\qquad\text{for $m=0,\dots,n$.}
\end{equation}
\end{lemma}

\begin{proof}[Proof of Proposition~\ref{p-approximation}]
Since $f^{(0)}$ is harmonic in a simply-connected domain, there indeed exists an analytic function $g(w)$ such that $f^{(0)}=g+\overline{g}$. Since the Hessian of $f^{(0)}$ is non-vanishing, it follows that $g''(w)$ is, because $4|g''(w)|^2=4f^{(0)}_{ww}f^{(0)}_{\bar w\bar w}=\left(f^{(0)}_{xx}\right)^2+\left(f^{(0)}_{xy}\right)^2\ne 0$. Thus, there indeed exists a continuous branch of $\sqrt{g''(w)}$, hence its antiderivative $h(z)$. 


By Lemma~\ref{l-complex-poisson}, it suffices to show that $C^2$ functions 
$f^{(0)}(w)=2\mathrm{Re}\,g(w)$, $f^{(1)}(w)=|h(w)|^{2}$, and
$f^{(2)}(w)=2\,\mathrm{Re}\!\left(h(w)^{2}\right)\log|h'(w)|$
satisfy~\eqref{eq-complex-poisson}, where $(h'(w))^{2}=g''(w)$ (a continuous branch is fixed).

Since $g$ is complex analytic, it follows that $f^{(0)}_{w\bar w}=0$, hence  $f^{(0)}(w)$  satisfies~\eqref{eq-complex-poisson} for $m=0$.
Next, 
\[
f^{(1)}_{w\bar w}=h'\bar h'=|g''|=\sqrt{f^{(0)}_{ww}f^{(0)}_{\bar w\bar w}},
\]
thus $f^{(1)}(w)$ satisfies~\eqref{eq-complex-poisson} for $m=1$.
Finally, $f^{(2)}(w)$ satisfies~\eqref{eq-complex-poisson} for $m=2$ because
\[
f^{(2)}_{w\bar w}
=\bar h\,\bar h'\,\frac{h''}{h'}+h\,h'\,\frac{\bar h''}{\bar h'}
=\frac{f^{(0)}_{ww}f^{(1)}_{\bar w\bar w}+f^{(1)}_{ww}f^{(0)}_{\bar w\bar w}}
      {\sqrt{f^{(0)}_{ww}f^{(0)}_{\bar w\bar w}}}.\\[-0.6cm]
\]
%
%
\end{proof}


Nothing like Proposition~\ref{p-approximation} 
is known for the approximation order higher than $2$, nor in Euclidean geometry.
However, formula~\eqref{eq:approx}  
is still applicable if the Hessian of $f^{(0)}(w)$ has a single zero of even ``multiplicity'', that is, $g''(w)$ has a single zero of even multiplicity. One just chooses the antiderivative $h(w)$ vanishing at this zero making $f^{(1)}(w)$ and $f^{(2)}(w)$ well-defined. 
For several zeros of even multiplicity, a reasonable approximation is obtained by replacing $\log|h'(w)|$ with $\log(|h'(w)|+t)$ in~\eqref{eq:approx}. For the design of examples, it is convenient to prescribe a polynomial $h'(w)$ first, then compute $h(w)$ and $g(w)$ by termwise integration. This method is realized in~\cite{yorov2025}. Let us illustrate it with two examples.


\begin{example}\label{ex:fig:5}
Let $f^{(0)}(w)=(w^4+\bar w^4)/12$ and $t=1/2$. Under the notation of Proposition~\ref{p-approximation}, we find $g(w)=w^4/12$, $h(w)=\int \sqrt{g''(w)}\,dw=w^2/2$, and $f(w,t)=(w^4+\bar w^4)(8+3\log|w|)/96+|w|^4/8$. 
Remeshing and optimization (neither of which is simple~\cite{yorov2025}) then leads to the asymptotic gridshell
in Fig.~\ref{fig:5}, right.
Since \(g''(w)\) has a zero of multiplicity \(2\) at \(w=0\), the flat point at the origin has valence \(8\).




    
\end{example}



\begin{example}\label{ex-no-xxtt} 
Let $f^{(0)}(w)=(w^3+\bar w^3)/2$ be the function from Example~\ref{ex-no-xxxt}. Under the notation of Proposition~\ref{p-approximation}, we find $g(w)=w^3/2$ and $h(w)=\int \sqrt{g''(w)}\,dw=\frac{2}{\sqrt{3}}w^{3/2}$.
Although $h(w)$ has no continuous branch in a vicinity of $w=0$, the expression $f(w,t)=(w^3+\bar w^3)(3+2t^2\log|3w|)/6+4|w|^3t/3$ given by~\eqref{eq:approx} is still well-defined. 
Notice that the coefficient at $t$ does not vanish for $w\ne0$.

To show the effect of the boundary condition, adjust the coefficient at $t$ by a harmonic polynomial:
\begin{align*}
f^{(1)}(w)&=\frac{4}{3}\left(|w|^3+\frac{1}{2}w^2+\frac{1}{2}\bar w^2-1\right).\\
\intertext{ 
The resulting function still satisfies~\eqref{eq-complex-poisson}.
It vanishes for $|w|^3+w^2/2+\bar w^2/2-1=0$,
which is the equation of an analytic Jordan curve.
Thus $f^{(1)}$ is the unique solution of~\eqref{eq-complex-poisson} for $m=1$ in the domain bounded by the curve, vanishing on the boundary.
\endgraf
Now by~\eqref{eq-complex-poisson} for $m=2$ we compute}
f^{(2)}_{w\bar w}(w)&=\frac{4}{3}\cdot\frac{w+\bar w}{|w|}+\frac{w^3+\bar w^3}{|w|^2} \qquad\text{for }w\ne 0,\\
\intertext{
which is discontinuous at the origin. In other words, $\left.f_{xxtt}+f_{yytt}\right|_{t=0}$ is discontinuous at the origin. 
\endgraf
We can still find a continuous solution, which is not $C^2$:
}
f^{(2)}(w)&=\frac{16}{9}(w+\bar w)|w|+\frac{2}{3}(w^3+\bar w^3)\log|w|+\text{a harmonic function},
\end{align*}
where the second summand is extended by $0$ at the origin, and the third one is found by solving the Laplace equation with the Dirichlet boundary condition. We observe more singular behavior here.
\end{example}

The idea of starting with an analog in isotropic geometry goes far beyond Plateau's problem. 
See, for instance, \cite{isometric-isotropic, pirahmad2024area,pirahmad2025,yorov2025} for an analogous approach to isometries of surfaces, both smooth and discrete.

\bigskip

\textbf{Acknowledgements.} The authors are grateful to 
I.C.~Dolcetta, G.~Di Fazio, D.~Gomes, S.~Melikhov, E.~Morozov, 
A.~Sageman-Furnas, J.M.~Urbano for useful discussions.

\bibliography{sn-bibliography}

\def\Yu{Yu}

\begin{thebibliography}{51}
\ifx \bisbn   \undefined \def \bisbn  #1{ISBN #1}\fi
\ifx \binits  \undefined \def \binits#1{#1}\fi
\ifx \bauthor  \undefined \def \bauthor#1{#1}\fi
\ifx \batitle  \undefined \def \batitle#1{#1}\fi
\ifx \bjtitle  \undefined \def \bjtitle#1{#1}\fi
\ifx \bvolume  \undefined \def \bvolume#1{\textbf{#1}}\fi
\ifx \byear  \undefined \def \byear#1{#1}\fi
\ifx \bissue  \undefined \def \bissue#1{#1}\fi
\ifx \bfpage  \undefined \def \bfpage#1{#1}\fi
\ifx \blpage  \undefined \def \blpage #1{#1}\fi
\ifx \burl  \undefined \def \burl#1{\textsf{#1}}\fi
\ifx \doiurl  \undefined \def \doiurl#1{\url{https://doi.org/#1}}\fi
\ifx \betal  \undefined \def \betal{\textit{et al.}}\fi
\ifx \binstitute  \undefined \def \binstitute#1{#1}\fi
\ifx \binstitutionaled  \undefined \def \binstitutionaled#1{#1}\fi
\ifx \bctitle  \undefined \def \bctitle#1{#1}\fi
\ifx \beditor  \undefined \def \beditor#1{#1}\fi
\ifx \bpublisher  \undefined \def \bpublisher#1{#1}\fi
\ifx \bbtitle  \undefined \def \bbtitle#1{#1}\fi
\ifx \bedition  \undefined \def \bedition#1{#1}\fi
\ifx \bseriesno  \undefined \def \bseriesno#1{#1}\fi
\ifx \blocation  \undefined \def \blocation#1{#1}\fi
\ifx \bsertitle  \undefined \def \bsertitle#1{#1}\fi
\ifx \bsnm \undefined \def \bsnm#1{#1}\fi
\ifx \bsuffix \undefined \def \bsuffix#1{#1}\fi
\ifx \bparticle \undefined \def \bparticle#1{#1}\fi
\ifx \barticle \undefined \def \barticle#1{#1}\fi
\bibcommenthead
\ifx \bconfdate \undefined \def \bconfdate #1{#1}\fi
\ifx \botherref \undefined \def \botherref #1{#1}\fi
\ifx \url \undefined \def \url#1{\textsf{#1}}\fi
\ifx \bchapter \undefined \def \bchapter#1{#1}\fi
\ifx \bbook \undefined \def \bbook#1{#1}\fi
\ifx \bcomment \undefined \def \bcomment#1{#1}\fi
\ifx \oauthor \undefined \def \oauthor#1{#1}\fi
\ifx \citeauthoryear \undefined \def \citeauthoryear#1{#1}\fi
\ifx \endbibitem  \undefined \def \endbibitem {}\fi
\ifx \bconflocation  \undefined \def \bconflocation#1{#1}\fi
\ifx \arxivurl  \undefined \def \arxivurl#1{\textsf{#1}}\fi
\csname PreBibitemsHook\endcsname

\bibitem[\protect\citeauthoryear{Schling}{2018}]{schling:2018}
\begin{botherref}
\oauthor{\bsnm{Schling}, \binits{E.}}:
Repetitive structures.
PhD thesis,
TU Munich
(2018)
\end{botherref}
\endbibitem

\bibitem[\protect\citeauthoryear{Schling et~al.}{}]{schling-aag-2018}
\begin{botherref}
\oauthor{\bsnm{Schling}, \binits{E.}},
\oauthor{\bsnm{Kilian}, \binits{M.}},
\oauthor{\bsnm{Wang}, \binits{H.}},
\oauthor{\bsnm{Schikore}, \binits{D.}},
\oauthor{\bsnm{Pottmann}, \binits{H.}}:
Design and construction of curved support structures with repetitive
  parameters.
In: Advances in Architectural Geometry 2018,
pp. 140--165.
Klein Publishing Ltd
\end{botherref}
\endbibitem

\bibitem[\protect\citeauthoryear{Wang and Pottmann}{2022}]{HelmutHui2022}
\begin{barticle}
\bauthor{\bsnm{Wang}, \binits{H.}},
\bauthor{\bsnm{Pottmann}, \binits{H.}}:
\batitle{Characteristic parameterizations of surfaces with a constant ratio of
  principal curvatures}.
\bjtitle{Computer Aided Geometric Design}
\bvolume{93},
\bfpage{102074}
(\byear{2022})
\end{barticle}
\endbibitem

\bibitem[\protect\citeauthoryear{Sachs}{1990}]{sachs}
\begin{bbook}
\bauthor{\bsnm{Sachs}, \binits{H.}}:
\bbtitle{Isotrope {G}eometrie des {R}aumes}.
\bpublisher{Vieweg, Braunschweig/Wiesbaden},
\blocation{Wiesbaden}
(\byear{1990})
\end{bbook}
\endbibitem

\bibitem[\protect\citeauthoryear{Gilbarg and
  Trudinger}{1983}]{Gilbarg-Trudinger-83}
\begin{bbook}
\bauthor{\bsnm{Gilbarg}, \binits{D.}},
\bauthor{\bsnm{Trudinger}, \binits{N.}}:
\bbtitle{Elliptic Partial Differential Equations of Second Order}.
\bpublisher{Springer},
\blocation{New York}
(\byear{1983})
\end{bbook}
\endbibitem

\bibitem[\protect\citeauthoryear{Nitsche}{1975}]{nitsche-1975}
\begin{bbook}
\bauthor{\bsnm{Nitsche}, \binits{J.C.C.}}:
\bbtitle{{V}orlesungen über {M}inimalflächen}.
\bsertitle{Die Grundlehren der Mathematischen Wissenschaften},
vol. \bseriesno{199},
p. \bfpage{775}.
\bpublisher{Springer},
\blocation{Berlin--Heidelberg--New York}
(\byear{1975}).
\doiurl{10.1007/978-3-642-65619-4}
\end{bbook}
\endbibitem

\bibitem[\protect\citeauthoryear{Dierkes et~al.}{1992}]{dierkes-1992-II}
\begin{bbook}
\bauthor{\bsnm{Dierkes}, \binits{U.}},
\bauthor{\bsnm{Hildebrandt}, \binits{S.}},
\bauthor{\bsnm{K{\"u}ster}, \binits{A.}},
\bauthor{\bsnm{Wohlrab}, \binits{O.}}:
\bbtitle{Minimal Surfaces II}.
\bsertitle{{G}rundlehren der {M}athematischen {W}issenschaften},
vol. \bseriesno{296},
p. \bfpage{422}.
\bpublisher{Springer},
\blocation{Berlin Heidelberg}
(\byear{1992}).
\doiurl{10.1007/978-3-662-08776-3}
\end{bbook}
\endbibitem

\bibitem[\protect\citeauthoryear{Bernstein}{1927}]{Bernstein-27}
\begin{barticle}
\bauthor{\bsnm{Bernstein}, \binits{S.}}:
\batitle{Sur l'int\'egration des \'equations aux d\'eriv\'ees partielles du
  type elliptique, {II}, {N}ote}.
\bjtitle{Math. Ann.}
\bvolume{96},
\bfpage{633}--\blpage{647}
(\byear{1927})
\end{barticle}
\endbibitem

\bibitem[\protect\citeauthoryear{Korn}{1909}]{Korn-09}
\begin{botherref}
\oauthor{\bsnm{Korn}, \binits{A.}}:
{\"U}ber {M}inimalfl\"achen, deren {R}andkurven wenig von ebenen {K}urven
  abweichen.
Abhandlungen Preuss. Akad.
(1909)
\end{botherref}
\endbibitem

\bibitem[\protect\citeauthoryear{M\"untz}{1911}]{Muentz-11}
\begin{barticle}
\bauthor{\bsnm{M\"untz}, \binits{C.}}:
\batitle{Zum {R}andwertproblem der partiellen {D}ifferentialgleichung der
  {M}inimalfl\"achen}.
\bjtitle{J. Reine Angew. Math.}
\bvolume{139},
\bfpage{52}--\blpage{79}
(\byear{1911})
\end{barticle}
\endbibitem

\bibitem[\protect\citeauthoryear{Douglas}{1931}]{Douglas31}
\begin{barticle}
\bauthor{\bsnm{Douglas}, \binits{J.}}:
\batitle{Solution of the problem of {P}lateau}.
\bjtitle{Trans. Amer. Math. Soc.}
\bvolume{33},
\bfpage{263}--\blpage{321}
(\byear{1931})
\end{barticle}
\endbibitem

\bibitem[\protect\citeauthoryear{R\'ado}{1930}]{Rado30}
\begin{barticle}
\bauthor{\bsnm{R\'ado}, \binits{T.}}:
\batitle{On {P}lateau's problem}.
\bjtitle{Ann. of Math. (2)}
\bvolume{31},
\bfpage{457}--\blpage{469}
(\byear{1930})
\end{barticle}
\endbibitem

\bibitem[\protect\citeauthoryear{Hardt and Simon}{1979}]{HardtSimon79}
\begin{barticle}
\bauthor{\bsnm{Hardt}, \binits{R.}},
\bauthor{\bsnm{Simon}, \binits{L.}}:
\batitle{Boundary regularity and embedded solutions for the oriented {P}lateau
  problem}.
\bjtitle{Ann. of Math. (2)}
\bvolume{110},
\bfpage{439}--\blpage{486}
(\byear{1979})
\end{barticle}
\endbibitem

\bibitem[\protect\citeauthoryear{Friedman}{1958}]{Friedman-58}
\begin{barticle}
\bauthor{\bsnm{Friedman}, \binits{A.}}:
\batitle{On the regularity of the solutions of nonlinear elliptic and parabolic
  systems of partial differential equations}.
\bjtitle{Journal of Mathematics and Mechanics}
\bvolume{7},
\bfpage{43}--\blpage{59}
(\byear{1958})
\end{barticle}
\endbibitem

\bibitem[\protect\citeauthoryear{Morrey and
  Nirenberg}{1957}]{Morrey-Nirenberg-1957}
\begin{barticle}
\bauthor{\bsnm{Morrey}, \binits{J.} \bsuffix{C.~B.}},
\bauthor{\bsnm{Nirenberg}, \binits{L.}}:
\batitle{{On the analyticity of the solutions of linear elliptic systems of
  partial differential equations}}.
\bjtitle{Communications on Pure and Applied Mathematics}
\bvolume{10},
\bfpage{271}--\blpage{290}
(\byear{1957})
\doiurl{10.1002/cpa.3160100204}
\end{barticle}
\endbibitem

\bibitem[\protect\citeauthoryear{Petrowsky}{1996}]{petrowsky-1996-i}
\begin{bbook}
\bauthor{\bsnm{Petrowsky}, \binits{I.G.}}:
\bbtitle{Selected Works. Part I: Systems of Partial Differential Equations and
  Algebraic Geometry}.
\bsertitle{Classics of Soviet Mathematics},
vol. \bseriesno{5}.
\bpublisher{Gordon and Breach Publishers},
\blocation{Amsterdam}
(\byear{1996})
\end{bbook}
\endbibitem

\bibitem[\protect\citeauthoryear{Serrin}{1970}]{Serrin70}
\begin{barticle}
\bauthor{\bsnm{Serrin}, \binits{J.}}:
\batitle{The {D}irichlet problem for surfaces of constant mean curvature}.
\bjtitle{Proc. London Math. Soc. (3)}
\bvolume{21},
\bfpage{361}--\blpage{384}
(\byear{1970})
\end{barticle}
\endbibitem

\bibitem[\protect\citeauthoryear{L\'opez}{2013}]{LopezBook13}
\begin{bbook}
\bauthor{\bsnm{L\'opez}, \binits{R.}}:
\bbtitle{Constant Mean Curvature Surfaces with Boundary}.
\bsertitle{Springer Monographs in Mathematics}.
\bpublisher{Springer},
\blocation{Berlin--Heidelberg}
(\byear{2013})
\end{bbook}
\endbibitem

\bibitem[\protect\citeauthoryear{Hildebrandt}{1970}]{Hildebrandt70}
\begin{barticle}
\bauthor{\bsnm{Hildebrandt}, \binits{S.}}:
\batitle{On the {P}lateau problem for surfaces of constant mean curvature}.
\bjtitle{Commun. Pure Appl. Math.}
\bvolume{23},
\bfpage{97}--\blpage{114}
(\byear{1970})
\end{barticle}
\endbibitem

\bibitem[\protect\citeauthoryear{Gulliver and Spruck}{1971}]{GulliverSpruck71}
\begin{barticle}
\bauthor{\bsnm{Gulliver}, \binits{R.}},
\bauthor{\bsnm{Spruck}, \binits{J.}}:
\batitle{The {P}lateau problem for surfaces of prescribed mean curvature in a
  cylinder}.
\bjtitle{Invent. Math.}
\bvolume{13},
\bfpage{169}--\blpage{178}
(\byear{1971})
\end{barticle}
\endbibitem

\bibitem[\protect\citeauthoryear{Koiso}{2002}]{Koiso2002}
\begin{barticle}
\bauthor{\bsnm{Koiso}, \binits{M.}}:
\batitle{Deformation and stability of surfaces with constant mean curvature}.
\bjtitle{Tohoku Math. J. (2)}
\bvolume{54},
\bfpage{145}--\blpage{159}
(\byear{2002})
\end{barticle}
\endbibitem

\bibitem[\protect\citeauthoryear{Figalli}{2017}]{Figalli-2017}
\begin{bbook}
\bauthor{\bsnm{Figalli}, \binits{A.}}:
\bbtitle{The Monge--Amp{\`e}re Equation and Its Applications}.
\bsertitle{Zurich Lectures in Advanced Mathematics}.
\bpublisher{European Mathematical Society},
\blocation{Z{\"u}rich}
(\byear{2017}).
\doiurl{10.4171/170}
\end{bbook}
\endbibitem

\bibitem[\protect\citeauthoryear{Caffarelli and
  Cabr{\'e}}{1995}]{caffarelli-cabre-1995}
\begin{bbook}
\bauthor{\bsnm{Caffarelli}, \binits{L.A.}},
\bauthor{\bsnm{Cabr{\'e}}, \binits{X.}}:
\bbtitle{Fully Nonlinear Elliptic Equations}.
\bsertitle{Colloquium Publications},
vol. \bseriesno{43},
p. \bfpage{104}.
\bpublisher{American Mathematical Society},
\blocation{Providence, RI}
(\byear{1995})
\end{bbook}
\endbibitem

\bibitem[\protect\citeauthoryear{Guan and Spruck}{2002}]{GuanSpruck2002}
\begin{barticle}
\bauthor{\bsnm{Guan}, \binits{B.}},
\bauthor{\bsnm{Spruck}, \binits{J.}}:
\batitle{The existence of hypersurfaces of constant {G}auss curvature with
  prescribed boundary}.
\bjtitle{J. Differential Geom.}
\bvolume{62},
\bfpage{259}--\blpage{287}
(\byear{2002})
\end{barticle}
\endbibitem

\bibitem[\protect\citeauthoryear{Caffarelli et~al.}{1988}]{caffarelli-1988}
\begin{barticle}
\bauthor{\bsnm{Caffarelli}, \binits{L.A.}},
\bauthor{\bsnm{Nirenberg}, \binits{L.}},
\bauthor{\bsnm{Spruck}, \binits{J.}}:
\batitle{Nonlinear {S}econd-{O}rder {E}lliptic {E}quations {V}. {The Dirichlet
  Problem for Weingarten Hypersurfaces}}.
\bjtitle{Communications on Pure and Applied Mathematics}
\bvolume{41},
\bfpage{41}--\blpage{70}
(\byear{1988})
\doiurl{10.1002/cpa.3160410103}
\end{barticle}
\endbibitem

\bibitem[\protect\citeauthoryear{Lin and Trudinger}{1994}]{lin-1994}
\begin{barticle}
\bauthor{\bsnm{Lin}, \binits{M.}},
\bauthor{\bsnm{Trudinger}, \binits{N.S.}}:
\batitle{The {D}irichlet problem for the prescribed curvature quotient
  equations}.
\bjtitle{Topological Methods in Nonlinear Analysis}
\bvolume{3},
\bfpage{307}--\blpage{323}
(\byear{1994})
\end{barticle}
\endbibitem

\bibitem[\protect\citeauthoryear{Jiao and Sun}{2022}]{jiao-sun-2022}
\begin{botherref}
\oauthor{\bsnm{Jiao}, \binits{H.}},
\oauthor{\bsnm{Sun}, \binits{Z.}}:
The {D}irichlet problem for a class of prescribed curvature equations.
Journal of Geometric Analysis
\textbf{32}
(2022)
\doiurl{10.1007/s12220-022-00994-0}
\end{botherref}
\endbibitem

\bibitem[\protect\citeauthoryear{Hopf}{1951}]{hopf-1951}
\begin{barticle}
\bauthor{\bsnm{Hopf}, \binits{H.}}:
\batitle{{\"U}ber {F}l\"achen mit einer {R}elation zwischen den
  {H}auptkr\"ummungen}.
\bjtitle{Math. Nachr.}
\bvolume{4},
\bfpage{232}--\blpage{249}
(\byear{1951})
\end{barticle}
\endbibitem

\bibitem[\protect\citeauthoryear{K\"uhnel}{2013}]{kuehnel-2013}
\begin{bbook}
\bauthor{\bsnm{K\"uhnel}, \binits{W.}}:
\bbtitle{Differentialgeometrie}.
\bsertitle{Aufbaukurs Mathematik}.
\bpublisher{Springer},
\blocation{Wiesbaden}
(\byear{2013})
\end{bbook}
\endbibitem

\bibitem[\protect\citeauthoryear{Mladenov and Oprea}{2003}]{mladenov+2003}
\begin{barticle}
\bauthor{\bsnm{Mladenov}, \binits{I.M.}},
\bauthor{\bsnm{Oprea}, \binits{J.}}:
\batitle{The {M}ylar balloon revisited}.
\bjtitle{Amer. Math. Monthly}
\bvolume{110}(\bissue{9}),
\bfpage{761}--\blpage{784}
(\byear{2003})
\end{barticle}
\endbibitem

\bibitem[\protect\citeauthoryear{Mladenov and Oprea}{2007}]{mladenov+2007}
\begin{barticle}
\bauthor{\bsnm{Mladenov}, \binits{I.M.}},
\bauthor{\bsnm{Oprea}, \binits{J.}}:
\batitle{The {M}ylar balloon: new viewpoints and generalizations}.
\bjtitle{Geom. Integrability \& Quantization}
\bvolume{8},
\bfpage{246}--\blpage{263}
(\byear{2007})
\end{barticle}
\endbibitem

\bibitem[\protect\citeauthoryear{L{\'o}pez and
  P{\'a}mpano}{2020}]{lopez-pampano-2020}
\begin{barticle}
\bauthor{\bsnm{L{\'o}pez}, \binits{R.}},
\bauthor{\bsnm{P{\'a}mpano}, \binits{A.}}:
\batitle{Classification of rotational surfaces in {E}uclidean space satisfying
  a linear relation between their principal curvatures}.
\bjtitle{Math. Nachrichten}
\bvolume{293},
\bfpage{735}--\blpage{753}
(\byear{2020})
\end{barticle}
\endbibitem

\bibitem[\protect\citeauthoryear{Liu et~al.}{2023}]{Liu-23}
\begin{botherref}
\oauthor{\bsnm{Liu}, \binits{Y.}},
\oauthor{\bsnm{Pirahmad}, \binits{O.}},
\oauthor{\bsnm{Wang}, \binits{H.}},
\oauthor{\bsnm{Michels}, \binits{D.L.}},
\oauthor{\bsnm{Pottmann}, \binits{H.}}:
Helical surfaces with a constant ratio of principal curvatures.
Beitr. Algebra Geom.
(2023).
\url{https://doi.org/10.1007/s13366-022-00670-y}
\end{botherref}
\endbibitem

\bibitem[\protect\citeauthoryear{Riveros and Corro}{2012}]{riveros+2012}
\begin{barticle}
\bauthor{\bsnm{Riveros}, \binits{C.M.C.}},
\bauthor{\bsnm{Corro}, \binits{A.M.V.}}:
\batitle{Surfaces with constant {C}hebyshev angle}.
\bjtitle{Tokyo J. Math.}
\bvolume{35}(\bissue{2}),
\bfpage{359}--\blpage{366}
(\byear{2012})
\end{barticle}
\endbibitem

\bibitem[\protect\citeauthoryear{Riveros and Corro}{2013}]{riveros+2013}
\begin{barticle}
\bauthor{\bsnm{Riveros}, \binits{C.M.C.}},
\bauthor{\bsnm{Corro}, \binits{A.M.V.}}:
\batitle{Surfaces with constant {C}hebyshev angle {II}}.
\bjtitle{Tokyo J. Math.}
\bvolume{36}(\bissue{2}),
\bfpage{379}--\blpage{386}
(\byear{2013})
\end{barticle}
\endbibitem

\bibitem[\protect\citeauthoryear{St{\"a}ckel}{1896}]{staeckel-1896}
\begin{botherref}
\oauthor{\bsnm{St{\"a}ckel}, \binits{P.}}:
Beitr{\"a}ge zur {F}l{\"a}chentheorie. {III}.\ {Z}ur {T}heorie der
  {M}inimalfl{\"a}chen.
Leipziger Berichte,
491--497
(1896)
\end{botherref}
\endbibitem

\bibitem[\protect\citeauthoryear{L{\'o}pez}{2008}]{lopez-2008}
\begin{barticle}
\bauthor{\bsnm{L{\'o}pez}, \binits{R.}}:
\batitle{Linear {W}eingarten surfaces in {E}uclidean and hyperbolic space}.
\bjtitle{Matem{\'a}tica Contempor{\^a}nea}
\bvolume{35},
\bfpage{95}--\blpage{113}
(\byear{2008})
\doiurl{10.21711/231766362008/rmc356}
\end{barticle}
\endbibitem

\bibitem[\protect\citeauthoryear{Yorov
  et~al.}{2025}]{Yorov-Pottmann-Skopenkov-23}
\begin{barticle}
\bauthor{\bsnm{Yorov}, \binits{K.}},
\bauthor{\bsnm{Skopenkov}, \binits{M.}},
\bauthor{\bsnm{Pottmann}, \binits{H.}}:
\batitle{Surfaces of constant principal-curvatures ratio in isotropic
  geometry}.
\bjtitle{Contrib. Algebra Geom.}
\bvolume{66},
\bfpage{775}--\blpage{814}
(\byear{2025})
\doiurl{10.1007/s13366-024-00768-5}
\end{barticle}
\endbibitem

\bibitem[\protect\citeauthoryear{M{\"u}ller and
  Pottmann}{2025}]{isometric-isotropic}
\begin{botherref}
\oauthor{\bsnm{M{\"u}ller}, \binits{C.}},
\oauthor{\bsnm{Pottmann}, \binits{H.}}:
{Isometric Surfaces in Isotropic 3-Space}.
arXiv
(2025)
{\href{https://arxiv.org/abs/2504.11351}{{arXiv:2504.11351}}}
{[math.DG]}.
40 pages
\end{botherref}
\endbibitem

\bibitem[\protect\citeauthoryear{G{\'a}lvez et~al.}{2004}]{galvez-2004}
\begin{barticle}
\bauthor{\bsnm{G{\'a}lvez}, \binits{J.A.}},
\bauthor{\bsnm{Mart{\'\i}nez}, \binits{A.}},
\bauthor{\bsnm{Mil{\'a}n}, \binits{F.}}:
\batitle{Complete {L}inear {W}eingarten {S}urfaces of {B}ryant {T}ype. {A}
  {P}lateau problem at {I}nfinity}.
\bjtitle{Transactions of the American Mathematical Society}
\bvolume{356}(\bissue{9}),
\bfpage{3405}--\blpage{3428}
(\byear{2004})
\doiurl{10.1090/S0002-9947-04-03592-5}
\end{barticle}
\endbibitem

\bibitem[\protect\citeauthoryear{Yorov et~al.}{2025}]{yorov2025}
\begin{botherref}
\oauthor{\bsnm{Yorov}, \binits{K.}},
\oauthor{\bsnm{Wang}, \binits{B.}},
\oauthor{\bsnm{Skopenkov}, \binits{M.}},
\oauthor{\bsnm{Pottmann}, \binits{H.}},
\oauthor{\bsnm{Jiang}, \binits{C.}}:
Solving {E}uclidean Problems by Isotropic Initialization
(2025).
\url{https://arxiv.org/abs/2506.01726}
\end{botherref}
\endbibitem

\bibitem[\protect\citeauthoryear{Strubecker}{1941}]{strubecker:1941}
\begin{barticle}
\bauthor{\bsnm{Strubecker}, \binits{K.}}:
\batitle{Differentialgeometrie des isotropen {R}aumes {I}: {T}heorie der
  {R}aumkurven}.
\bjtitle{Sitzungsber. Akad. Wiss. Wien, Abt. IIa}
\bvolume{150},
\bfpage{1}--\blpage{53}
(\byear{1941})
\end{barticle}
\endbibitem

\bibitem[\protect\citeauthoryear{Strubecker}{1942a}]{strubecker:1942}
\begin{barticle}
\bauthor{\bsnm{Strubecker}, \binits{K.}}:
\batitle{Differentialgeometrie des isotropen {R}aumes {III}:
  {F}l{\"a}chentheorie}.
\bjtitle{Math. Zeitschrift}
\bvolume{48},
\bfpage{369}--\blpage{427}
(\byear{1942})
\end{barticle}
\endbibitem

\bibitem[\protect\citeauthoryear{Strubecker}{1942b}]{strubecker:1942a}
\begin{barticle}
\bauthor{\bsnm{Strubecker}, \binits{K.}}:
\batitle{Differentialgeometrie des isotropen {R}aumes {II}: {D}ie {F}l{\"a}chen
  konstanter {R}elativkr{\"u}mmung {$K=rt-s^2$}}.
\bjtitle{Math. Zeitschrift}
\bvolume{47},
\bfpage{743}--\blpage{777}
(\byear{1942})
\end{barticle}
\endbibitem

\bibitem[\protect\citeauthoryear{Ara{\'u}jo et~al.}{2017}]{Urbano-2017}
\begin{barticle}
\bauthor{\bsnm{Ara{\'u}jo}, \binits{D.J.}},
\bauthor{\bsnm{Teixeira}, \binits{E.V.}},
\bauthor{\bsnm{Urbano}, \binits{J.M.}}:
\batitle{A proof of the {$C^{p'}$}-regularity conjecture in the plane}.
\bjtitle{Advances in Mathematics}
\bvolume{316},
\bfpage{541}--\blpage{553}
(\byear{2017})
\doiurl{10.1016/j.aim.2017.06.027}
\end{barticle}
\endbibitem

\bibitem[\protect\citeauthoryear{Sawano et~al.}{2020}]{Di-Fazio}
\begin{bbook}
\bauthor{\bsnm{Sawano}, \binits{Y.}},
\bauthor{\bsnm{Di~Fazio}, \binits{G.}},
\bauthor{\bsnm{Hakim}, \binits{D.I.}}:
\bbtitle{Morrey Spaces: Introduction and Applications to Integral Operators and
  PDE's}.
\bpublisher{Chapman \& Hall/CRC},
\blocation{New York}
(\byear{2020})
\end{bbook}
\endbibitem

\bibitem[\protect\citeauthoryear{K\"onig}{2011}]{Koenig-11}
\begin{barticle}
\bauthor{\bsnm{K\"onig}, \binits{M.}}:
\batitle{On {J. Schauder's} method to solve elliptic differential equations}.
\bjtitle{J. Fixed Point Theory Appl.}
\bvolume{9},
\bfpage{135}--\blpage{196}
(\byear{2011})
\end{barticle}
\endbibitem

\bibitem[\protect\citeauthoryear{Agmon et~al.}{1959}]{Agmon-etal-59}
\begin{barticle}
\bauthor{\bsnm{Agmon}, \binits{S.}},
\bauthor{\bsnm{Douglis}, \binits{A.}},
\bauthor{\bsnm{Nirenberg}, \binits{L.}}:
\batitle{Estimates near the boundary for solutions of elliptic p.d.e.
  satisfying a general boundary value condition {I}}.
\bjtitle{Commun. Pure Appl. Math.}
\bvolume{12},
\bfpage{623}--\blpage{727}
(\byear{1959})
\end{barticle}
\endbibitem

\bibitem[\protect\citeauthoryear{Munkres}{1966}]{munkres-1966}
\begin{bbook}
\bauthor{\bsnm{Munkres}, \binits{J.R.}}:
\bbtitle{Elementary Differential Topology}.
\bsertitle{Annals of Mathematics Studies},
vol. \bseriesno{54}.
\bpublisher{Princeton University Press},
\blocation{Princeton, NJ}
(\byear{1966}).
\doiurl{10.1515/9781400882656} .
\bcomment{Revised edition; original 1963}
\end{bbook}
\endbibitem

\bibitem[\protect\citeauthoryear{Pirahmad et~al.}{2025a}]{pirahmad2024area}
\begin{botherref}
\oauthor{\bsnm{Pirahmad}, \binits{O.}},
\oauthor{\bsnm{Pottmann}, \binits{H.}},
\oauthor{\bsnm{Skopenkov}, \binits{M.}}:
Area preserving {C}ombescure transformations.
Results Math
\textbf{80}(27)
(2025)
\end{botherref}
\endbibitem

\bibitem[\protect\citeauthoryear{Pirahmad et~al.}{2025b}]{pirahmad2025}
\begin{botherref}
\oauthor{\bsnm{Pirahmad}, \binits{O.}},
\oauthor{\bsnm{Pottmann}, \binits{H.}},
\oauthor{\bsnm{Skopenkov}, \binits{M.}}:
Flexible polyhedral nets in isotropic geometry.
arXiv preprint arXiv:2504.15060
(2025)
\end{botherref}
\endbibitem

\end{thebibliography}



\end{document}